\documentclass[11pt]{amsart}

\usepackage{a4wide}
\usepackage{amssymb}
\usepackage{mathrsfs}
\usepackage{enumerate}
\usepackage{esint}
\usepackage{titletoc}
\usepackage[colorlinks=true, urlcolor=red, linkcolor=red, citecolor=blue]{hyperref}
\usepackage[nameinlink]{cleveref}
\usepackage[normalem]{ulem} 
\usepackage{xcolor}

\theoremstyle{plain}
\newtheorem{theorem}{Theorem}[section]
\newtheorem{lemma}[theorem]{Lemma}
\newtheorem{corollary}[theorem]{Corollary}
\newtheorem{proposition}[theorem]{Proposition}

\theoremstyle{definition}

\newtheorem{remark}[theorem]{Remark}

\numberwithin{equation}{section}

\DeclareMathOperator{\DIV}{div}

\newcommand{\R}{{\mathbb{R}}}

\newcommand{\cA}{{\mathcal{A}}}

\makeatletter
\@namedef{subjclassname@2020}{\textup{2020} Mathematics Subject Classification}
\makeatother

\title[Gradient estimates for Orlicz growth equations]{Gradient estimates for nonlinear elliptic equations with  Orlicz growth and measure data}
\author{Ying Li, Chao Zhang$^*$}
\address{Ying Li\newline
	School of Mathematics, Harbin Institute of Technology, Harbin 150001, China\newline
	\texttt{lymath@hit.edu.cn}}
\address{Chao Zhang\newline
	School of Mathematics and Institute for Advanced Study in Mathematics, Harbin Institute of Technology, Harbin 150001, China
	\newline
	\texttt{czhangmath@hit.edu.cn}}
\thanks{$^*$ Corresponding author.}

\begin{document}
	\begin{abstract}
		We  establish gradient estimates of solutions to a class of nonlinear elliptic equations with measure data under Orlicz-type growth conditions. 
		The growth is governed by the structural condition
		\[
		0<i_a\le t g'(t)/g(t)\le s_a<1.
		\]
		We obtain two types of regularity results: pointwise Wolff potential estimates for the gradient of solutions in the singular regime $i_a \in \big(\frac{n-1}{2n-1},1\big)$, and Lipschitz regularity of the solutions in the  regime $i_a \in (0,1)$. 
		In the power-type case $g(t)=t^{p-1}$, our results recover the known gradient estimates for the singular $p$-Laplace equation.
	\end{abstract}

	\maketitle
	\section{Introduction}\label{sec1}
	We  investigate gradient estimates for solutions to the following quasilinear elliptic equation with measure data under Orlicz-type growth:
	\begin{equation}\label{eq-main}
		-\DIV\cA(x,Du)=\mu,
	\end{equation}
	in a bounded open subset $\Omega$ of $\R^n$, where $n\ge2$ and $\mu$ is  an arbitrary bounded measure on $\Omega$.  The vector field $\mathcal A:\Omega\times\mathbb R^n\to\mathbb R^n$ is assumed to be
	continuously differentiable with respect to the gradient variable and to satisfy the
	following structural conditions: for all $x,x_1,x_2\in\Omega$ and $\xi,\eta\in\mathbb R^n$,
	\begin{equation}\label{eq:struct3}
		\begin{split}
			|\mathcal A(x,\xi)|+|D_\xi\mathcal A(x,\xi)|\,|\xi|
			&\le \Lambda\, g(|\xi|), \\
			\langle D_\xi\mathcal A(x,\xi)\eta,\eta\rangle
			&\ge \lambda \, g'(|\xi|)\,|\eta|^2,\\
			|\mathcal A(x_1,\xi)-\mathcal A(x_2,\xi)|
			&\le w(|x_1-x_2|)\, g(|\xi|),
	\end{split}	\end{equation}
	with $0<\lambda\leq 1\leq \Lambda$, and 
	$w:[0,\infty)\to[0,1]$ is a nondecreasing, subadditive modulus of continuity with
	\begin{equation}\label{eq:w-cond}
		w(0)=0, \qquad \lim_{\rho\to0}w(\rho)=0,
	\end{equation}
	and satisfies the Dini-type integrability condition
	\begin{equation}\label{eq:dini}
		\int_0^R
		\frac{w(\rho)^{\frac{2i_a}{(1+i_a)^2}}}{\rho}\,d\rho<\infty
		\qquad \text{for every } R<\infty.
	\end{equation}
	Throughout this paper,  $g:[0,\infty)\to[0,\infty)$ is a $C^1$-function satisfying
	\begin{equation}\label{eq:g-cond}
		0<i_a:=\inf_{t>0}\frac{t g'(t)}{g(t)}
		\le \sup_{t>0}\frac{t g'(t)}{g(t)}=:s_a<1, \qquad G(t):=\int_0^t g(\tau)\,d\tau,
	\end{equation}
	where $G$ is an $N$-function (see Section~\ref{sec2}). 
	In addition, we work under the assumption that there exists constant $c>0$ such that for all $t\geq0$,
	\begin{equation}\label{eq-0}
		g(t)\geq c\,t^{i_a}.
	\end{equation}
	
	We see that typical examples of $g$ include the power-type growth $g(t)=t^{p-1}$ with $1<p<2$, the borderline logarithmic growth $g(t)=t^{p-1}\log(e+t)$ with $1<p<2$, and the double-phase growth $g(t)=t^{p-1}+t^{q-1}$ with $1<p\le q<2$.
	
	The study of potential theory can be traced back to~\cite{M70,MHbook}. A fundamental step in nonlinear potential theory was made by  Kilpel\"{a}inen and Mal\'{y}~\cite{KM94}, who derived pointwise Wolff potential estimates for  solutions   to $p$-Laplace equations with nonnegative measures.   Subsequently, alternative proofs and extensions of these estimates have been developed by Trudinger and Wang~\cite{TW02,TW09}, allowing one to treat more general subelliptic operators. Further proofs of pointwise potential estimates can be found in \cite{DM-ameriacn-11,KK10}, where an approach covering the case of general signed measures was developed. For more results concerning pointwise estimates of solutions, we refer to~\cite{BS,CGZ24,CMM,Kim-jfa-25,NNSW,NOS}.

	
	The first extension of pointwise potential estimates of solutions  to the corresponding gradient estimates was obtained in \cite{Mingione11}, where Riesz potential estimates for gradients of solutions to the $p$-Laplace equation were established in the linear case $p=2$. For $p>2$, Wolff potential estimates for gradients of solutions were  established in \cite{DM-ameriacn-11}.  These results were subsequently refined in \cite{KM13}, where gradient estimates were further expressed in terms of Riesz potentials.  Moreover, Duzaar and Mingione \cite{DM-jfa-10} obtained Riesz potential estimates for gradients in the subquadratic range $p \in (2-\frac{1}{n},\,2)$. Later, Baroni \cite{BP-calc-15} extended these results to elliptic equations with Orlicz growth, establishing Riesz potential estimates for gradients of solutions of the form
	\[
	g(\left|\nabla u(x_0)\right|) \leq C g\!\left(\fint_{B(x_0,2R)} |\nabla u| \,dx \right)
	+ C \mathcal{I}_1^{2R}(|\mu|)(x),
	\]
	under the structural assumption
	\[
	1 \le i_a := \inf_{t>0} \frac{t g'(t)}{g(t)}
	\le \sup_{t>0} \frac{t g'(t)}{g(t)} =: s_a,
	\]
	where
	\[
	\mathcal{I}_1^{2R}(|\mu|)(x) := \int^{2R}_0 \frac{|\mu|(B_{\rho})}{\rho^{n-1}}\,\frac{d\rho}{\rho}
	\]
	denotes the Riesz potential of the measure $|\mu|$.  In addition, Wolff potential estimates  for the Orlicz growth equations under the assumptions that $1-\frac{1}{n}<i_a\leq s_a<1$ were studied by Yao and Zhou in~\cite{YZ26}. We refer the readers to the comprehensive survey~\cite{KM14} for  further background on potential estimates and their applications.

	In a more recent work, Nguyen and Phuc \cite{NP20} investigated the singular range
	\[\frac{3n-2}{2n-1} < p \le 2-\frac{1}{n},\]
	for the $p$-Laplace equation with Dini-continuous coefficients, and established Wolff potential estimates for gradients of solutions.
	Dong and Zhu \cite{Dongz-jems-24} subsequently strengthened the gradient estimates by replacing Wolff potentials with Riesz potentials, and further obtained Lipschitz regularity of solutions for $p\in(1,2)$. For the lower range $p\in(1,\frac{3}{2})$, Riesz potential estimates for gradients were obtained in \cite{NP23}. We also mention that related gradient estimates  for $p$-Laplace type equations with more general Dini-continuous coefficients can be found in \cite{xu}.  As potential estimates plays an important  role in regularity theory, the study of gradient potential estimates has attracted significant interest and has been extended to a variety of equations~\cite{BSY,CMM,CSYZ,DKLN-b,DKLN,Km14-b,XZM}.

	The papers mentioned above motivate the study of analogous questions for elliptic equations with Orlicz growth in the regime $i_a<1-\frac{1}{n}$.  However, as shown in~\cite{DM-jfa-10}, when $0<i_a<1-\frac1n$, a distributional solution to \eqref{eq-main} with bounded measure data may fail to belong to $W^{1,1}_{\mathrm{loc}}(\Omega)$.   To address this, a weaker notion of derivative based on truncations is considered, following the approach introduced in~\cite{NP20}.   For $k>0$ and $s\in\mathbb R$, the standard truncation is defined by
	\[
	T_k(s):=\max\{-k,\min\{s,k\}\}.
	\]
	If $u$ satisfies $T_k(u)\in W_0^{1,G}(\Omega)$ for every $k>0$, then there exists a measurable vector field $Z_u:\Omega\to\mathbb R^n$ such that
	\[
	\nabla T_k(u)=\chi_{\{|u|<k\}}\,Z_u \quad \text{a.e. in }\Omega,
	\]
	see~\cite[Lemma~2.1]{BBGP95}. The vector field $Z_u$ is referred to as the \emph{generalized gradient} of $u$ and denoted by $\nabla u$. We point out here that if $u$ is a distributional solution to the Dirichlet problem
	\begin{align}\label{eq:sec4}
		\begin{cases}
			- \DIV\left(\cA(x,\nabla u)\right) = \mu & \text{in } \Omega,\\
			u = 0 & \text{on } \partial \Omega,
		\end{cases}
	\end{align}
	such that for every $k>0$, the truncated function $T_k(u)$ belongs to $W_0^{1,G}(\Omega)$ and satisfies
	\[
	- \DIV\big(\cA(x,\nabla T_k(u))\big) = \mu_k \quad \text{in } \mathcal{D}'(\Omega)
	\]
	for some finite measure $\mu_k$ with $\mu_k \to \mu$ and $|\mu_k| \to |\mu|$ weakly in the sense of measures on $\mathbb{R}^n$, 
	then the pointwise estimate \eqref{ineq:int} holds almost everywhere for the generalized gradient $\nabla u$.   In particular, any \emph{renormalized solution} to \eqref{eq:sec4} satisfies the above  conditions, so the pointwise estimate always applies. 
	For the existence of renormalized solutions to elliptic equations with Orlicz growth and measure data, we refer to \cite{C23}. See also \cite{DMOP} for several equivalent definitions of renormalized solutions.
	
	Our  main results  in this manuscript consist  two types of regularity estimates for  solutions to \eqref{eq-main}:  pointwise Wolff potential estimates for the  gradient   in the singular regime $i_a \in \big(\frac{n-1}{2n-1},1\big)$,  and Lipschitz regularity in the  regime $i_a \in (0,1)$. Since the gradient of solutions $\nabla u$ may not belong to $L^1_{\mathrm{loc}}(\Omega)$, we need to  replace the standard mean oscillation 
	\[\fint_{B_{\rho}}|\nabla u-(\nabla u)_{\rho}|\,dx \] 
	by the quantity
	\[ \phi(x,\rho)=\left(\inf_{q\in \mathbb{R}^n}\fint_{B_{\rho}}|\nabla u-q|^{\gamma}\,dx\right)^{\frac{1}{\gamma}},\quad \gamma<1.\]
	Such modifications have been carried out in \cite{dong-arma-12,DLK20,DK17,DZ22,NP20,NP23}. To establish decay estimates for the excess functional $\phi(x,\rho)$ (Propisition~\ref{prop1}), it is crucial to obtain comparison estimates of the form $\left(\fint_{B_{2r}} (\nabla u - \nabla w)^{\gamma}\,dx \right)^{\frac{1}{\gamma}}$, where $w$ is a solution to the homogeneous equation
	\[-\DIV\left(\cA(x,\nabla w)\right)=0\]
	with  condition $w=u$ on $\partial B_{2r}$(see Lemma~\ref{lem-32}). We remark that in Lemma ~\ref{lem-32}, when $i_a \in(\frac{n-1}{2n-1},1)$, the right-hand side $\gamma_0$ can be taken up to $1-i_a$, while for $i_a \in (0, \frac{n-1}{2n-1})$ only $\gamma_0 < 1-i_a$ is attainable. This is the reason in the case $i_a \in (0, \frac{n-1}{2n-1})$, we can only derive Lipschitz estimates for solutions. Our arguments are much influenced by \cite{dong-arma-12,NP20}. However, in contrast to the $p$-Laplace case, a more delicate analysis is required due to the non-homogeneity of the operator, which prevents the direct use of standard scaling arguments.  Moreover, we need a reverse H\"{o}lder inequality for $\nabla w$ (Lemma~\ref{lemma-RH}), which, to the best of our knowledge, has not appeared in the literature. Achieving this requires not only adapting the classical iteration lemma to the Orlicz framework (Lemma~\ref{lemma-61}), but also carefully combining several additional auxiliary lemmas. In addition, the non-homogeneous nature of Orlicz growth requires an extra structural condition \eqref{eq-0} to obtain the desired decay estimates for the excess functional $\phi(x,\rho)$ (Propisition~\ref{prop1}). 
	
	A (weak) solution to \eqref{eq-main} is a function $u\in W_{\rm loc}^{1,G}(\Omega)$ such that
	\begin{equation}
		\int_{\Omega}\cA(x,\nabla u)\cdot\nabla \varphi \,dx=\int_{\Omega}\varphi \,d\mu
	\end{equation}
	holds for all $\varphi\in C_0^{\infty}(\Omega)$.

	Our main results are stated as follows.
	\begin{theorem}[Interior pointwise gradient estimate]\label{thm:int}
		Let $i_a\in \big(\frac{n-1}{2n-1},1\big)$. If $u\in W^{1,G}_{\mathrm{loc}}(\Omega)$ is a solution to \eqref{eq-main},
		and assume that \eqref{eq:struct3}--\eqref{eq-0} hold.  
		Then there exists a constant
		\[
		C = C(n,i_a,s_a,\lambda,\Lambda, \omega),
		\]
		such that for every Lebesgue point $x\in\Omega$ of the vector--valued function
		$\nabla u$ and for any $B_R(x)\subset\Omega$ with $R\in(0,1]$, the pointwise estimate
		\begin{equation}\label{ineq:int}
			|\nabla u(x)|
			\le
			C\Big[\mathbf{W}^{R}_{i_a,g}(|\mu|)(x)\Big]^{\frac1{i_a}}
			+
			C\left(
			\fint_{B_R(x)} \left(|\nabla u(y)|+1\right)^{1-i_a}\,dy
			\right)^{\frac1{1-i_a}}
		\end{equation}
		holds almost everywhere. Here the truncated Wolff potential is defined by
		\[
		\mathbf{W}^{R}_{i_a,g}(|\mu|)(x)
		:=
		\int_0^{R}
		\left(
		g^{-1}\!\left(
		\frac{|\mu|(B_\rho(x))}{\rho^{\,n-1}}
		\right)
		\right)^{i_a}
		\frac{d\rho}{\rho}.
		\]
	\end{theorem}
	
	\begin{theorem}[Interior Lipschitz estimate]\label{thm:int:lip}
		Let $i_a \in(0,1)$. If $u\in W^{1,G}_{\rm{loc}}(\Omega)$ is a solution to \eqref{eq-main} and \eqref{eq:struct3}--\eqref{eq-0} hold, then there exists a constant $C=C(n, i_a,s_a, \Lambda, \lambda,\omega)$ such that the estimate
		\begin{equation}\label{ineq:int:lip}
			\|\nabla u\|_{L^\infty(B_{R/2}(x))}
			\leq C \big\|\mathbf{W}_{i_a,g}^R(|\mu|)\big\|^\frac{1}{i_a}_{L^\infty(B_{R}(x))}
			+ C R^{-\frac{n}{1-i_a}} \||\nabla u|+1\|_{L^{1-i_a}(B_{R}(x))}
		\end{equation}
		holds for any $B_R(x)\subset\Omega$ with $R\in (0, 1]$.
	\end{theorem}
	
	\begin{remark}
		In the power--type case $g(t)=t^{p-1}$ with $1<p<2$, the structural conditions
		\eqref{eq:struct3}--\eqref{eq-0} are satisfied with
		\[
		i_a = s_a = p-1.
		\]
		In this situation, equation \eqref{eq-main} reduces to the $p$--Laplace equation
		\[
		-\DIV\big(a(x)|\nabla u|^{p-2}\nabla u\big)=\mu,
		\]
		where $a(x)$ is a Dini continuous function in $\Omega$. Moreover, in this case the truncated Orlicz--Wolff potential
		$\mathbf{W}^R_{i_a,g}$ is exactly the truncated Riesz potential
		\[
		\mathcal{I}^{R}_1(|\mu|)(x)
		=
		\int_0^R
		\left(
		\frac{|\mu|(B_\rho(x))}{\rho^{\,n-1}}
		\right)
		\frac{d\rho}{\rho}.
		\]
		Consequently, Theorems~\ref{thm:int} and~\ref{thm:int:lip} recover the interior 	pointwise gradient estimate and the interior Lipschitz estimate for solutions to the $p$--Laplace equation with measure data obtained in \cite{Dongz-jems-24}.
	\end{remark}

	The paper is organized as follows. In Section \ref{sec2}, we give some preliminary lemmas.  The proof of Theorem~\ref{thm:int} is in  Section  \ref{sec3}.  Section \ref{sec4} is devoted to the proofs of Theorem~\ref{thm:int:lip}.

	\section{Preliminart Lemmas}\label{sec2}
	In this section we list some definitions and preliminary lemmas. If not specified, a constant $C$ is a positive constant, possibly changing line by line. We denote by $(\cdot)_{B_R(x)}$ the average over the ball $B_R(x)$.

	A function $G:[0,\infty)\to [0,\infty)$ is called a \emph{Young function}  if it has the form 
	\begin{equation}\label{eq18}
		G(t)=\int_0^t g(\tau)\,d\tau
	\end{equation}
	for some non-decreasing, continuous function $g:[0,\infty)\to [0,\infty)$ that is not identically zero. A Young function $G$ is an \emph{$N$-function} if $0 < G(t) < \infty$ for all $t>0$ and
	\[
	\lim_{t\to \infty} \frac{G(t)}{t} = \lim_{t\to 0^+} \frac{t}{G(t)} = \infty.
	\] 
	The \emph{conjugate function} $G^\ast$ of an $N$-function $G$ is defined by
	\[
	G^\ast(\tau) = \sup_{t \ge 0} \bigl( t\tau - G(t) \bigr), \quad \tau \ge 0.
	\]
	It is well-known that $G^\ast$ is also an $N$-function, and Young's inequality holds:
	\begin{equation}\label{eq-Young}
		t\tau \le \epsilon G(t) + C(\epsilon) G^\ast(\tau), \quad \text{for all } t,\tau \ge 0 \text{ and any } \epsilon>0.
	\end{equation}
	Moreover, the conjugate function satisfies
	\[
	G^\ast(g(t)) \le C\, G(t).
	\]
	A Young function $G$ is said to satisfy the \(\Delta_2\)-condition if
	\[
	G(2t) \le C\, G(t) \quad \text{for some } C>0,
	\]
	and the \(\nabla_2\)-condition if
	\[
	G(t) \le \frac{G(\theta t)}{2\theta} \quad \text{for some } \theta>1 \text{ and all } t>0.
	\]
	
	From \cite[Proposition~2.9]{cinch-cpde-11} and \cite[Lemma~1.9]{yao}, we know that $G\in \Delta_2 \cap \nabla_2$. The following lemma, which can be found in \cite[Lemma~2.1]{Kim-jfa-25}, describes the scaling properties of $G$ and $g^{-1}$.
	
	\begin{lemma}\label{lem:Gg}
		Assume that $G$ and $g$ satisfy \eqref{eq:g-cond}. Then, for all $\lambda \le 1$,
		\[
		\lambda^{s_a+1} G(t) \le G(\lambda t) \le \lambda^{i_a+1} G(t), \quad
		c \,\lambda^{1/i_a} g^{-1}(t) \le g^{-1}(\lambda t) \le c^{-1} \lambda^{1/s_a} g^{-1}(t).
		\]
		For all $\lambda\ge 1$,
		\[
		\lambda^{i_a+1} G(t) \le G(\lambda t) \le \lambda^{s_a+1} G(t), \quad
		c \,\lambda^{1/s_a} g^{-1}(t) \le g^{-1}(\lambda t) \le c^{-1} \lambda^{1/i_a} g^{-1}(t).
		\]
	\end{lemma}

	The Orlicz space $L^G(\Omega)$ is defined as 
	\begin{equation*}
		L^{G}(\Omega):=\{u:\Omega \to \mathbb{R}~\text{measurable}:\varrho_G(u)< \infty \},
	\end{equation*}
	endowed with the Luxemburg norm 
	\begin{equation*}
		\|u\|_{L^G(\Omega)}:=\inf\left\{\lambda>0:\varrho_{G}\left(\frac{|u|}{\lambda}\right)\leq 1\right\},
	\end{equation*}
	where
	\begin{equation*}
		\varrho_{G}(u):= \int_{\Omega}G(|u|)\,dx
	\end{equation*}
	is called $\varrho_{G}$-modular. We define the Orlicz-Sobolev space $W^{1,G}(\Omega)$ as follows:
	\begin{equation*}
		W^{1,G}(\Omega)=\left\{u\in W_{\rm{loc}}^{1,1}(\Omega):|u|,|\nabla u|\in L^G(\Omega)\right\},
	\end{equation*}
	where the gradient is understood in the distributional sense. The norm on this space is given by  
	\begin{equation*}
		\|u\|_{W^{1,G}(\Omega)}:=\inf\left\{\lambda>0:\varrho_{G}\left(\frac{|u|}{\lambda}\right)+\varrho_{G}\left(\frac{|\nabla u|}{\lambda}\right)\leq 1
	\right\}.
	\end{equation*}
	The space $W_0^{1,G}(\Omega)$ denotes the closure of $C^{\infty}_0(\Omega)$ with respect to this norm.

	The following lemma comes from \cite[Lemma~2.2]{YZ26}.
	\begin{lemma}
		Assume that $\cA$ satisfies \emph{(1.2)--(1.7)} with $s_a<1$, and let $G$ be defined in \eqref{eq18}. Then there exist constants $C=C(n,i_a,s_a)>0$,  such that the following holds for all
		$\xi,\eta\in\R^n$.
		\medskip
		\noindent
		\textnormal{(i)}
		For every $\varepsilon\in(0,1)$,
		\begin{equation}\label{eq-2.3}
			\begin{split}
				G(|\xi-\eta|)&\leq 	C\,\varepsilon^{\frac{i_a-1}{1+i_a}}\frac{g(|\xi|+|\eta|)}{|\xi|+|\eta|}\,|\xi-\eta|^2 + \varepsilon G(|\eta|) \\
				&	\le C\,\varepsilon^{\frac{i_a-1}{1+i_a}}
				\bigl(\cA(x,\xi)-\cA(x,\eta)\bigr)\cdot(\xi-\eta)
				+\varepsilon\,G(|\eta|),
		\end{split}\end{equation}
		and
		\begin{equation}\label{2.4}
			\bigl(\cA(x,\xi)-\cA(x,\eta)\bigr)\cdot(\xi-\eta)
			\ge
			C\,|V(\xi)-V(\eta)|^2,
		\end{equation}
		where
		\[
		V(z):=\sqrt{\frac{g(|z|)}{|z|}}\,z .
		\]
		
		\medskip
		\noindent
		\textnormal{(ii)}
		We have
		\begin{equation}\label{2.5}
			C\,\frac{g(|\xi|+|\eta|)}{|\xi|+|\eta|}\,|\xi-\eta|^2
			\le
			|V(\xi)-V(\eta)|^2
			\le
			C\,\frac{g(|\xi|+|\eta|)}{|\xi|+|\eta|}\,|\xi-\eta|^2,
		\end{equation}
		and
		\begin{equation}\label{2.6}
			C\,\sqrt{\frac{g(|\xi|+|\eta|)}{|\xi|+|\eta|}}\,|\xi-\eta|^2
			\le
			\bigl(V(\xi)-V(\eta)\bigr)\cdot(\xi-\eta)
			\le
			C\,\sqrt{\frac{g(|\xi|+|\eta|)}{|\xi|+|\eta|}}\,|\xi-\eta|^2.
		\end{equation}
	\end{lemma}
	
	The proof of the following iteration lemma  in the standard case can be found in~\cite[Lemma 6.1]{book-GE-direct}. For the sake of completeness, we include a proof in the setting of Orlicz-growth equations. 
	\begin{lemma}\label{lemma-61}
		Let $\rho_1>\rho>0$ with $B_{\rho_1}\Subset \Omega$, $\xi \in W_{\rm loc}^{1,G}(\Omega)$, and let $G$ be defined as above.  
		Assume for all $\rho \le t < s \le \rho_1$, we have
		\[
		\int_{B_t}G(\nabla \xi)\,dx \le \theta \int_{B_s}G(|\nabla \xi|)\,dx +\int_{B_{s}}G\left(\frac{A}{s-t}\right)\,dx.
		\]
		with $A\ge 0$ and $\theta\in(0,1)$.  Then there exists a constant $C>0$, depending on $\theta$ and $s_a$, such that
		\[
		\int_{B_{\rho}}G(\nabla \xi)\,dx \leq C\int_{B_{\rho_1}}G\left(\frac{A}{\rho_1-\rho}\right)\,dx.
		\]
	\end{lemma}
	\begin{proof}
		Fix a parameter $\tau<1$, to be chosen later.
		Define a sequence 
		\[
		t_k:=\rho+(1-\tau^{k})(\rho_1-\rho), \qquad k=0,1,2,\cdots.
		\]
		Then $t_0=\rho$, $t_k\to \rho_1$, and
		\[
		t_{k+1}-t_k=\tau^{k}(1-\tau)(\rho_1-\rho).
		\]
		Applying the assmution with $t=t_k$ and $s=t_{k+1}$, we obtain
		\begin{equation*}
			\int_{B_{t_{k}}}G(|\nabla \xi|)\,dx \leq  \theta\int_{B_{t_{k+1}}}G(|\nabla \xi
			|)\,dx+c\int_{B_{t_{k+1}}}G\!\left(\frac{A}{\tau^k(1-\tau)(\rho_1-\rho)}\right)\,dx.
		\end{equation*}			
		Denote
		\[Z(t_{k})=\int_{B_{t_k}}G(|\nabla \xi|)\,dx.
		\]
		Iterating the above inequality yields, for any $m\in\mathbb N$,
		\[
		Z(\rho)=Z(t_0)
		\le C\sum_{k=0}^{m-1}\theta^k
		\int_{B_{\rho_1}}G\!\left(\frac{A\tau^{-k}}{(1-\tau)(\rho_1-\rho)}\right)\,dx
		+\theta^m Z(t_m).
		\]
		Since $Z$ is bounded and $\theta\in(0,1)$, letting $m\to\infty$ gives
		\[
		Z(\rho)\le C\sum_{k=0}^{\infty}\theta^k\int_{B_{\rho_1}}
		G\!\left(\frac{\tau^{-k}A}{(1-\tau)(\rho_1-\rho)}\right)\,dx.
		\]
		
		By Lemma~\ref{lem:Gg},	we  known that there exists a constant $C>0$, depending only on $s_a$,
		such that
		\[
		G(\lambda t)\le C\,\lambda^{s_a+1}G(t)
		\qquad\text{for all }\lambda\ge1,\ t>0.
		\]
		Consequently,
		\[
		G\!\left(\frac{\tau^{-k}A}{(1-\tau)(\rho_1-\rho)}\right)
		\le \frac{\tau^{-k(1+s_a)}}{(1-\tau)^{1+s_a}}
		G\!\left(\frac{A}{\rho_1-\rho}\right).
		\]
		Therefore,
		\[
		Z(\rho)\le \frac{C}{(1-\tau)^{1+s_a}}\sum_{k=0}^{\infty}\bigl(\theta\,\tau^{-(s_a+1)}\bigr)^k\int_{B_{\rho_1}}G\!\left(\frac{A}{\rho_1-\rho}\right)\,dx.
		\]
		Choosing $\tau>\theta^{\frac{1}{s_a+1}}$ such that $\theta\,\tau^{-(s_a+1)}<1$, we deduce that 
		\[
		Z(\rho)=\int_{B_{\rho}}G(|\nabla \xi|)\,dx \le C\int_{B_{\rho_1}}G\!\left(\frac{A}{\rho_1-\rho}\right)\,dx,
		\]
		where $C$ depends only on $\theta$ and $s_a$.
	\end{proof}

Next, we present interior oscillation estimates for solutions to the homogeneous equation
\begin{equation}\label{eq-s2-1}
	-\DIV (\cA_0(\nabla v))=0 \quad \text{in }\Omega,
\end{equation}
where $\cA_0=\cA_0(\xi)$ is a vector field independent of $x$ and satisfying \eqref{eq:struct3}.

For the case $0<i_a<1$, a detailed proof of the following oscillation estimate can be found in 
\cite[Theorem~4.1]{A-26-ar}. 
When $i_a\ge 1$, we refer to \cite[Lemma~4.1]{BP-calc-15} and \cite[Theorem~3.1]{DM-ameriacn-11}. 
In the singular $p$--Laplacian regime $1<p<2$, see \cite[Theorem~3.3]{DM-jfa-10}. 
Estimates of this type, involved with different exponents, were originally developed in 
\cite{DM-amer-93, KZ-cpde-99, Lie-cpde-91}.

\begin{lemma}[Theorem~4.1.~\cite{A-26-ar}]\label{lem-24}
	Let $v\in W^{1,G}_{\mathrm{loc}}(\Omega)$ be a solution of \eqref{eq-s2-1}. 
	Then there exist constants $C>1$ and $\alpha\in(0,1)$ such that
	$v\in C^{1,\alpha}_{\mathrm{loc}}(\Omega)$. 
	Moreover, for every ball $B_R(x_0)\Subset\Omega$ and every $r\in(0,R)$, one has
	\[
	\fint_{B_r(x_0)}
	\big|\nabla v-(\nabla v)_{B_r(x_0)}\big|\,dx
	\le
	C\Big(\frac{r}{R}\Big)^{\alpha}
	\fint_{B_R(x_0)}
	\big|\nabla v-(\nabla v)_{B_R(x_0)}\big|\,dx,
	\]
	with $C$ and $\alpha$ depending only on $n,\Lambda,\lambda$, $i_a$ and $s_a$.
\end{lemma}

The previous lemma yields the following consequence. We refer to  \cite[Corollary~2.3]{Dongz-jems-24} for the details.

\begin{corollary}
	Under the assumptions of Lemma~\ref{lem-24}, there exist constants $C>1$ and 
	$\alpha\in(0,1)$, depending only on $n,\lambda,\Lambda, i_a$ and  $s_a$, such that 
	for every ball $B_R(x_0)\subset\Omega$ the estimate
	\begin{equation}\label{eq-sec2-231}
		R^{\alpha}[\nabla v]_{C^{\alpha}(B_{R/2}(x_0))}
		\le
		C\fint_{B_R(x_0)}
		\big|\nabla v-(\nabla v)_{B_R(x_0)}\big|\,dx
	\end{equation}
	holds. In addition, for every $r\in[R/2,R)$ we have
	\begin{equation}\label{eq-sec2-232}
		[\nabla v]_{C^{\alpha}(B_r(x_0))}
		\le
		C\frac{R^{\,n+1-\alpha}}{(R-r)^{n+1}}
		\fint_{B_R(x_0)}
		\big|\nabla v-(\nabla v)_{B_R(x_0)}\big|\,dx.
	\end{equation}
\end{corollary}

				By an argument analogous to that in~\cite[Theorem 2.1]{Dongz-jems-24}, we have the following lemma.
				\begin{lemma}\label{lem-dong21}
					Let $v \in W^{1,G}_{\mathrm{loc}}(\Omega)$ be a solution of \eqref{eq-s2-1} and let $\gamma_0\in(0,1)$. 
					Then there exist constants $\alpha\in(0,1)$ and $C>1$ such that
					$\alpha$ depends only on $n,\lambda,\Lambda,i_a,s_a$, while $C$ depends only on 
					$n,\lambda,\Lambda,i_a,s_a$ and $\gamma_0$, and for every ball 
					$B_R(x_0)\subset\Omega$ and every $\rho\in(0,R)$,
					\begin{equation}
						\inf_{q \in \mathbb{R}^n}
						\left(
						\fint_{B_\rho(x_0)} |\nabla v - q|^{\gamma_0}\,dx
						\right)^{1/\gamma_0}
						\le
						C \left( \frac{\rho}{R} \right)^{\alpha}
						\inf_{q \in \mathbb{R}^n}
						\left(
						\fint_{B_R(x_0)} |\nabla v - q|^{\gamma_0}\,dx
						\right)^{1/\gamma_0}.
					\end{equation}
				\end{lemma}

				Next we give a few miscellaneous and useful results.
				\begin{lemma}[Theorem 7,~\cite{DE-forum-08}]\label{lem:orlicz-poincare}
					Let $G$ be an $N$-function satisfying the $\Delta_2$-condition and $\nabla_2$-condition.  Then there exist constants $\eta \in (0,1)$,  such that for every $v \in W^{1,G}(B_R)$ it holds
					\begin{equation}\label{eq:orlicz-poincare}
						\fint_{B_R}G\left(\frac{|v - (v)_{B_R}|}{R}\right)\, dx
						\le C \left( \fint_{B_R} \big(G(|\nabla v|) \big)^\eta \, dx \right)^{1/\eta}.
					\end{equation}
				\end{lemma}

				\begin{lemma}[Proposition~6,~\cite{DE-forum-08}]\label{lem-book}
					Let $f\in L^1(\Omega)$. Suppose that for some $\eta\in (0,1)$, $c_1> 0$, and all ball $B_{\frac\rho2}$ with $B_{\rho}\subset \Omega$,
					\[\fint_{B_{\frac\rho2}}|f|\,dx \leq c_1 \left(\fint_{B_{\rho}}|f|^{\eta}\,dx\right)^{\frac{1}{\eta}}.\]
					Then there exists  $d_1>1$ and $c_2>1$ such that $f\in L^{d_1}(B_\rho)$, and for all $d_0\in [1,d_1]$,
					\[ \left(\fint_{B_{\frac\rho2}}|f|^{d_0}\,dx\right)^{\frac{1}{d_0}}\leq c_2\fint_{B_{\rho}}|f|\,dx. \]
				\end{lemma}

				\begin{lemma}[Lemma 3.1,~\cite{DM-ameriacn-11}]\label{lem-31dm}
					Let $f:\Omega\to \R^n$ be a integrable map such that
					\[\left(\fint_{B_{\frac{\rho}{2}}}|f|^{d_0}\,dx\right)^{\frac{1}{d_0}}\leq c \fint_{B_{\rho}}|f|\,dx\]
					holds whenever $B_{\rho}\subseteq\Omega$, and $d_0>1$, $c\geq0$. Then for every $t\in (0,1]$ and $d\in(0, d_0]$ there exists a constant $c_0(n,c,t)$ such that for every $B_{\rho}\Subset \Omega$, 
					\[\left(\fint_{B_{\frac{\rho}{2}}}|f|^{d}\,dx\right)^{\frac{1}{d}}\leq c_0 \left(\fint_{B_{\rho}}|f|^t\,dx\right)^{\frac{1}{t}}.
					\]
				\end{lemma}

				\section{Pointwise gradient estimates}\label{sec3}
				
				The section is devoted to the proof of pointwise gradient estimates for solutions of \eqref{eq-main}. Let $u$ be a solution to \eqref{eq-main} and $B_{2r}(x_0)\Subset\Omega$. We consider  the unique solution $w$ to the equation
				\begin{equation}\label{eq-sec3-1}
					\begin{cases}
						-\DIV \left( \cA(x,\nabla w )\right)=0 &\quad \mbox{in} \quad B_{2r}(x_0),\\
						w=u &\quad \mbox{on} \quad \partial B_{2r}(x_0).
					\end{cases}
				\end{equation}
				
				We shall derive the following estimate.
				
				\begin{lemma}\label{lem-32}
					Let $w$ be a solution to \eqref{eq-sec3-1}  and assume that $i_a\in (0,1)$. Then for any $\gamma_0\in (0,1-i_a]$ when $i_a\in (\frac{n-1}{2n-1},1)$, or $\gamma_0\in (0,\frac{ni_a}{n-1})$ when $i_a \in (0,\frac{n-1}{2n-1}]$, we have
					\begin{equation}
						\begin{split}
							&\left(  \fint_{B_{2r}(x_0)} |\nabla u-\nabla w|^{\gamma_0}\,dx\right)^{\frac{1}{\gamma_0}}\\
							&\quad \leq Cg^{-1}\left(\frac{|\mu|(B_{2r}(x_0))}{r^{n-1}}\right)+C\left( g^{-1}\left(\frac{|\mu|(B_{2r}(x_0))}{r^{n-1}} \right) \right)^{i_a}
							\fint_{B_{2r}(x_0)}(|\nabla u|+1)^{1-i_{a}}\,dx,
						\end{split}
					\end{equation}
					where $C$ is a constant depending only on $n$, $i_a$ and $\gamma_0$.
				\end{lemma}
				\begin{proof} 
					Here, we only give the proof in the case $i_a \in \big(\frac{n-1}{2n-1},1\big)$. The proof for $i_a\in (0,\frac{n-1}{2n-1}]$ is similar, with ideas inspired by \cite[Lemma~2.1]{NP22-b}.  Without loss of generality, we assume that $B_{2r}(x_0)=B_1$ and $|\mu|(B_{2r})=1$. Indeed, this can be achieved by setting
					\begin{equation}\begin{split}
							&\bar{u}(x)=\frac{u(x_0+2rx)}{2Ar},\qquad  \bar{\mu}(x)=\frac{2r\mu(x_0+2rx)}{g(A)},\\
							&\bar{\cA}(x,z)=\frac{\cA(x_0+2rx, Az)}{g(A)},\qquad A:=g^{-1}\left(\frac{|\mu|(B_{2r})}{(2r)^{n-1}}\right).
					\end{split}\end{equation}

					For $k>0$, using 
					\[\varphi_1= T_{2k}(u-w):=\max\{\min\{u-w, 2k\},-2k\}\]
					as  test function in \eqref{eq-main} and \eqref{eq-sec3-1} and recalling~\eqref{2.4}, we have 
					\begin{equation}\label{test1}
						\int_{C_k}|V(\nabla u)-V(\nabla w)|^2\,dx \leq Ck,
					\end{equation}
					where $C_k:=\{x\in B_1: |u(x)-w(x)|<2k\}$. 
					According to \eqref{eq:g-cond},  we know that $\frac{g(t)}{t}$ is strictly decreasing. Thus, by the triangle inequality, we have 
					\begin{equation*}
						\begin{split}
							|\nabla (u-w)|& \leq C\left(\frac{g(|\nabla u|+|\nabla w|)}{|\nabla u|+|\nabla w|}\right)^{-\frac{1}{2}}|V(\nabla u)-V(\nabla w)| \\
							&\leq C \left(\frac{g(|\nabla u|+|\nabla w|+1)}{|\nabla u|+|\nabla w|+1}\right)^{-\frac{1}{2}} \left| V(\nabla u)-V(\nabla w)\right|\\ 
							&\leq C \left(|\nabla u|+|\nabla w|+1\right)^{\frac{1-i_a}{2}} \left| V(\nabla u)-V(\nabla v)\right|\\
							&\leq C|V(\nabla u)-V(\nabla w)|(|\nabla u|+C|\nabla u-\nabla w|+1)^{\frac{1-i_a}{2}}\\
							&\leq C |V(\nabla u)-V(\nabla w)|(|\nabla u-\nabla w|)^{\frac{1-i_a}{2}}+C|V(\nabla u)-V(\nabla w)|(|\nabla u|+1)^{\frac{1-i_a}{2}}.
						\end{split}
					\end{equation*}
					Using Young's inequality, we find that 
					\begin{equation}\label{eq-sec3-34}
						|\nabla u-\nabla w| \leq C|V(\nabla u)-V(\nabla w)|^{\frac{2}{1+i_a}}+C|V(\nabla u)-V(\nabla w)|(|\nabla u|+1)^{\frac{1-i_a}{2}}.
					\end{equation}
					Now we set 
					\[ E_k=B_1\cap \{k<|u-w|<2k\}\quad \quad F_k=B_1\cap\{|u-w|>k\}.\]
					Using the Sobolev inequality, H\"{o}lder inequality and \eqref{eq-sec3-34}, we obtain 
					\begin{equation}\begin{split}
							&k\left|\{x:|u(x)-w(x)|>2k \}\cap B_1\right|^{\frac{n-1}{n}}\\
							&\quad \leq C\left( \int_{B_1}|T_{2k}(u-w)-T_{k}(u-w)|^{\frac{n}{n-1}}\,dx\right)^{\frac{n-1}{n}}\\
							&\quad \leq C\int_{E_k}|\nabla (u-w)|\,dx \\
							&\quad \leq C\int_{E_k}|V(\nabla u)-V(\nabla w)|^{\frac{2}{1+i_a}}+|V(\nabla u)-V(\nabla w)|(|\nabla u|+1)^{\frac{1-i_a}{2}}\,dx\\
							&\quad \leq C|E_{k}|^{\frac{i_a}{1+i_a}} \left(\int_{E_{k}}|V(\nabla u)-V(\nabla w)|^2\,dx\right)^{\frac{1}{1+i_a}}\\
							&\quad \quad +C\left(\int_{E_k}|V(\nabla v)-V(\nabla w)|^2\,dx\right)^{\frac{1}{2}}\left(\int_{E_k}(|\nabla u|+1)^{1-i_a}\,dx\right)^{\frac{1}{2}}.
						\end{split}
					\end{equation}
				It follows  from \eqref{test1} that
					\[
					k^{\frac12}|F_{2k}|^{\frac{n-1}{n}}
					\leq
					Ck^{-\frac12+\frac{1}{1+i_a}}|E_k|^{\frac{i_a}{1+i_a}}
					+ C Q_1^{\frac{1-i_a}{2}},
					\]
					where $Q_1:=\||\nabla u|+1\|_{L^{1-i_a}(B_1)}$. Taking the supremum over $k>0$ yields
					\begin{equation}
						\|u-w\|^{\frac12}_{L^{\frac{n}{2(n-1)},\infty}(B_1)}
						\leq
						C\|u-w\|^{\frac{1-i_a}{2(1+i_a)}}_{L^{\frac{1-i_a}{2i_a},\infty}(B_1)}
						+ C Q_1^{\frac{1-i_a}{2}} .
					\end{equation}
					Since $\frac{n-1}{2n-1}< i_a<1$, one has
					\[
					\frac{n}{2(n-1)}>\frac{1-i_a}{2i_a},
					\]
					and therefore
					\begin{equation}
						\|u-w\|^{\frac12}_{L^{\frac{n}{2(n-1)},\infty}(B_1)}
						\leq
						C\|u-w\|^{\frac{1-i_a}{2(1+i_a)}}_{L^{\frac{n}{2(n-1)},\infty}(B_1)}
						+ C Q_1^{\frac{1-i_a}{2}} .
					\end{equation}
					An application of Young's inequality gives
					\begin{equation}\label{ineq:weak}
						\|u-w\|^{\frac{1}{2}}_{L^{\frac{n}{2(n-1)},\infty}(B_1)}
						\leq
						C + C Q_1^{\frac{1-i_a}{2}} .
					\end{equation}
					Set
					\[
					H(u,w)=|V(\nabla u)-V(\nabla w)|^2,
					\qquad
					q=\frac{n}{2(n-1)} .
					\]
					For $k,l>0$, by Chebyshev's inequality and \eqref{test1}, we have
					\begin{align*}
						|\{H(u,w)>l\}\cap B_1|
						&\leq |\{|u-w|>k\}\cap B_1|
						+ |\{|u-w|\le k,\;H(u,w)>l\}\cap B_1| \\
						&\leq Ck^{-q}\|u-w\|^q_{L^{q,\infty}(B_1)}
						+ \frac{1}{l}\!\int_{B_1\cap\{|u-w|\le k\}} H(u,w)\,dx \\
						&\leq Ck^{-q}\|u-w\|^q_{L^{q,\infty}(B_1)} + \frac{Ck}{l}.
					\end{align*}
					Choosing
					\[
					k=\big[l\|u-w\|^q_{L^{q,\infty}(B_1)}\big]^{\frac{1}{1+q}},
					\]
					one obtains
					\[
					l^{\frac{q}{1+q}}
					|\{H(u,w)>l\}\cap B_1|
					\leq
					C\|u-w\|^{\frac{q}{1+q}}_{L^{q,\infty}(B_1)} .
					\]
					Taking the supremum over $l>0$ gives
					\begin{equation}\label{g1:weak}
						\|H(u,w)\|_{L^{\frac{q}{1+q},\infty}(B_1)}
						\leq
						C\|u-w\|_{L^{q,\infty}(B_1)} .
					\end{equation}
					
					Let $\gamma_0\in (0,1-i_a]$, then we have 
					\[
					\frac{\gamma_0}{1+i_a}<\frac{q}{1+q}.
					\]
					Thus, using \eqref{eq-sec3-34} and H\"older's inequality with exponents
					$\frac{2}{1-i_a}$ and $\frac{2}{1+i_a}$, we infer
					\begin{equation}\label{ineq:gamma0}
						\begin{aligned}
							&	\int_{B_1}|\nabla(u-w)|^{\gamma_0}\,dx\\
							&\quad\leq
							C\int_{B_1}
							\Big(
							H(u,w)^{\frac{\gamma_0}{1+i_a}}
							+ H(u,w)^{\frac{\gamma_0}{2}}
							(|\nabla u|+1)^{\frac{\gamma_0(1-i_a)}{2}}
							\Big) dx \\
							&\quad \leq C\int_{B_1}	H(u,w)^{\frac{\gamma_0}{1+i_a}}\,dx+C\left(\int_{B_1}	H(u,w)^{\frac{\gamma_0}{1+i_a}}\,dx\right)^{\frac{1+i_a}{2}}\left(\int_{B_1}(|\nabla u|+1)^{\gamma_0}\,dx
							\right)^{\frac{1-i_a}{2}}\\
							&\quad \leq
							C\|H(u,w)\|^{\frac{\gamma_0}{1+i_a}}_{L^{\frac{q}{1+q},\infty}(B_1)}
							+ C\|H(u,w)\|^{\frac{\gamma_0}{2}}_{L^{\frac{q}{1+q},\infty}(B_1)}
							Q_1^{\frac{\gamma_0(1-i_a)}{2}} .
						\end{aligned}
					\end{equation}
					Combining \eqref{ineq:weak}, \eqref{g1:weak} and \eqref{ineq:gamma0} yields
					\[
					\int_{B_1}|\nabla(u-w)|^{\gamma_0}\,dx
					\leq
					C + C Q_1^{\gamma_0(1-i_a)},
					\]
					which implies the desired result.
				\end{proof}

				\begin{lemma}  
					Assume that $w $ be a solution to \eqref{eq-sec3-1}. 
					If $v\in w+W_0^{1,G}(B_r(x_0))$ be the unique solution to
					\begin{equation}\label{eqa:v}
						\left\{
						\begin{aligned}
							-\DIV(\cA(x_0,\nabla v)) =&  0 \quad \text{in} \quad B_{r}(x_0), &\\
							v =&  w \quad \text{on}  \quad \partial B_{r}(x_0),&\\
						\end{aligned}
						\right.
					\end{equation}
					then  we have 
					\begin{equation}\label{eq-vw}
						\fint_{B_r(x_0)}G(|\nabla v-\nabla w|)\,dx\leq C\omega^{\frac{2i_a}{1+i_a}}(r)\fint_{B_{r}(x_0)}G(|\nabla w|)\,dx.
					\end{equation}
				\end{lemma}
				\begin{proof} 
					By testing  \eqref{eqa:v} with $w-v$, we obtain that 
					\begin{equation}
						\begin{split}
							C\int_{B_r(x_0)} G(|\nabla v|)\,dx &\leq \int_{B_r(x_0)} \cA(x_0,\nabla v)\cdot \nabla v\,dx= \int_{B_r(x_0)} \cA(x_0,\nabla v )\cdot \nabla w\,dx\\
							&\leq C\int_{B_r(x_0)}g(|\nabla v|)|\nabla w|\,dx\\
							&\leq \varepsilon\int_{B_r(x_0)} G^\ast(g(|\nabla v|))\,dx +C(\varepsilon)\int_{B_r(x_0)} G(|\nabla w|)\,dx,					
						\end{split}
					\end{equation}
					which implies that 
					\begin{equation}\label{eq-sec3-311}
						\int_{B_r(x_0)}G(|\nabla v|)\,dx \leq C\int_{B_r(x_0)}G(|\nabla w|)\,dx.
					\end{equation}
					By testing \eqref{eq-sec3-1} and \eqref{eqa:v} with $w-v$, we obtain the equality 
					\[I_1=I_2,\]
					where 
					\begin{equation}
						\begin{split}
							&I_1:= \int_{B_r(x_0)}\left( \cA(x,\nabla w)-\cA(x,\nabla v)\right)\cdot (\nabla w-\nabla v),\,\\
							&I_2:=-\int_{B_r(x_0)}\left(\cA(x,\nabla v)-\cA(x_0,\nabla v)\right)\cdot (\nabla w-\nabla v).
						\end{split}
					\end{equation}
					
					\textit{Estimate of $I_1$.} \eqref{eq-2.3} implies that 
					\[C\int_{B_r(x_0)}G(|\nabla w-\nabla v|)\,dx \leq C\varepsilon^{\frac{i_a-1}{1+i_a}}I_1+\varepsilon\int_{B_r(x_0)}G(|\nabla v|)\,dx.\]
					
					\textit{Estimate of $I_2$.} From \eqref{eq:struct3}, Young's inequality and \eqref{eq-sec3-311}, we obtain 
					\begin{equation}
						\begin{split}
							|I_2|&\leq \int_{B_r(x_0)} \left|\cA(x,\nabla v)-\cA(x_0,\nabla v)\right||\nabla w-\nabla v|\,dx\\
							&\leq\omega(r)\int_{B_r(x_0)} g(|\nabla v|)|\nabla w-\nabla v|\,dx\\
							&\leq  \omega(r)\int_{B_r(x_0)}G(|\nabla w-\nabla v|)\,dx +\omega(r) C\int_{B_r(x_0)} G^\ast(g(|\nabla v|))\,dx\\
							&\leq \omega(r)\int_{B_r(x_0)}G(|\nabla w-\nabla v|)\,dx +\omega(r)C \int_{B_r(x_0)} G(|\nabla v|)\,dx\\
							&\leq \omega(r)\int_{B_r(x_0)}G(|\nabla w-\nabla v|)\,dx +\omega(r)C\int_{B_r(x_0)} G(|\nabla w|)\,dx\\
							&\leq C\omega(r)\int_{B_r(x_0)} G(|\nabla w|)\,dx.
						\end{split}
					\end{equation}
					Choosing $\varepsilon=\omega(r)$, we combine the above two estimates of $I_1$ and $I_2$  to conclude that
					\begin{equation}
						\int_{B_r(x_0)}G(|\nabla w-\nabla v|)\,dx \leq C \omega^{\frac{2i_a}{i_a+1}}(r)\int_{B_r(x_0)}G(|\nabla w|)\,dx.
					\end{equation}
				\end{proof}
				
				\begin{lemma}[Reverse H\"older inequality]\label{lemma-RH}
					Let $w$ be a solution to \eqref{eq-sec3-1}. Then for every $t\in (0,1)$, there exists a constant 
					$C(n,t,i_a,s_a,\lambda,\Lambda)>0$ such that
					\begin{equation}\label{eq-RH}
						\fint_{B_{\frac{\rho}{2}}(y)} G(|\nabla w|)\,dx
						\le
						C\left(\fint_{B_{\rho}(y)} G^t(|\nabla w|)\,dx\right)^{\frac1t}
					\end{equation}
					holds for every ball $B_{\rho}(y)\subset B_{2r}(x_0)$.
				\end{lemma}
				
				\begin{proof}
					Let $y\in\Omega$ and $\rho>0$ with $B_{\rho}(y)\subset B_{2r}(x_0)$.
					Let $  \frac\rho2 \le t<s\le \rho$ and let $\eta\in C_0^\infty(B_s(y))$ be a cut--off function such that
					\[
					0\le \eta\le 1,\qquad 
					\eta\equiv 1 \ \text{in } B_t(y),\qquad 
					|\nabla \eta|\le \frac{C}{s-t}.
					\]
					Set
					\[
					\varphi=\eta\big(w-(w)_{B_{\rho}}\big), 
					\qquad 
					z:=w-\varphi=(w)_{B_{\rho}}+(1-\eta)\big(w-(w)_{B_{\rho}}\big).
					\]
					Then $\nabla\varphi=\nabla w$ in $B_t(y)$ and
					\[
					\nabla z=(1-\eta)\nabla w-(w-(w)_{B_{\rho}})\nabla\eta .
					\]

					Observe that
					\[
					\int_{B_t}G(|\nabla w|)\,dx =\int_{B_t}G(|\nabla \varphi|)\,dx
					\le \int_{B_s}G(|\nabla\varphi|)\,dx .
					\]
					Applying the $\Delta_2$-condition of $G$ together with $\nabla\varphi=\nabla w-\nabla z$ gives
					\[
					\int_{B_s}G(|\nabla\varphi|)
					\le c\int_{B_s}G(|\nabla w|)\,dx
					+c\int_{B_s}G(|\nabla z|)\,dx .
					\]
					Since $\varphi=w-z$ is an admissible test function in \eqref{eq-sec3-1}, arguing as in \eqref{eq-sec3-311} leads to
					\[
					\int_{B_s}G(|\nabla w|)\,dx
					\le c\int_{B_s}G(|\nabla z|)\,dx .
					\]
					Combining these estimates, we arrive at
					\begin{equation}\label{hf}
						\int_{B_t}G(|\nabla w|)\,dx
						\le c\int_{B_s}G(|\nabla z|)\,dx .
					\end{equation}
					Next we estimate $\nabla z$. By the definition of $z$ and the $\Delta_2$--condition,
					\[
					\int_{B_s}G(|\nabla z|)\,dx
					\le c\int_{B_s}G\big((1-\eta)|\nabla w|\big)\,dx
					+c\int_{B_s}G\big(|w-(w)_{B_{\rho}}||\nabla\eta|\big)\,dx .
					\]
					Since $(1-\eta)=0$ in $B_t$ and $|\nabla\eta|\le C/(s-t)$, we get
					\[
					\int_{B_s}G(|\nabla z|)\,dx
					\le c\int_{B_s\setminus B_t}G(|\nabla w|)\,dx
					+c\int_{B_s}G\!\left(\frac{|w-(w)_{B_{\rho}}|}{s-t}\right)\!dx .
					\]
					Inserting this into \eqref{hf} gives
					\[
					\int_{B_t}G(|\nabla w|)\,dx
					\le c\int_{B_s\setminus B_t}G(|\nabla w|)\,dx
					+c\int_{B_s}G\!\left(\frac{|w-(w)_{B_{\rho}}|}{s-t}\right)\!dx .
					\]
					
					Adding $c\int_{B_t}G(|\nabla w|)\,dx$ to both sides and dividing by $c+1$,
					we obtain the standard hole--filling inequality
					\begin{equation}\label{eq0}
						\int_{B_t}G(|\nabla w|)\,dx\le \frac{c}{c+1}\int_{B_s}G(|\nabla w|)\,dx
						+c\int_{B_s}G\!\left(\frac{|w-(w)_{B_{\rho}}|}{s-t}\right)\!dx.
					\end{equation}
					
						Applying Lemma~\ref{lemma-61} yields
					\[
					\int_{B_{\rho/2}}G(|\nabla w|)\,dx
					\le C\int_{B_\rho}G\!\left(\frac{|w-(w)_{B_\rho}|}{\rho}\right)\,dx .
					\]
	We infer from Lemma~\ref{lem:orlicz-poincare} that there exists
					$\widetilde\theta\in(0,1)$, such that
					\[
					\fint_{B_{\rho/2}} G(|\nabla w|)\,dx\le C\fint_{B_\rho}G\!\left(\frac{|w-(w)_{B_\rho}|}{\rho}\right)\,dx
					\le
					C\Big(\fint_{B_\rho} G^{\widetilde\theta}(|\nabla w|)\,dx\Big)^{1/\widetilde\theta}.
					\]
	Lemma~\ref{lem-book} then implies the existence of a constant \(d_0>1\) for which
					\[
					\Big(\fint_{B_{\rho/2}} G^{d_0}(|\nabla w|)\,dx\Big)^{1/d_0}
					\le
					C\fint_{B_{\rho}} G(|\nabla w|)\,dx .
					\]
By Lemma~\ref{lem-31dm}, it follows that for every
					$t\in(0,1]$ there exists a constant $C>0$ such that
					\[
					\fint_{B_{\rho/2}} G(|\nabla w|)\,dx
					\le
					C\Big(\fint_{B_\rho} G^t(|\nabla w|)\,dx\Big)^{1/t},
					\]
					where $C>0$ depends on $\{t,n, i_a, s_a,\lambda,\Lambda\}$.

				\end{proof}
				
				For a ball $B_\rho(x)\Subset\Omega$ and a function $f\in  W_{\text{loc}}^{1,G}(\Omega)$, there exists $\mathbf{q}_{x,\rho}(f)\in \R^n$ such that
				\begin{equation*}
					\left(\fint_{B_\rho(x)}|\nabla f-\mathbf{q}_{x,\rho}(f)|^{\gamma_0}\,dx\right)^{1/\gamma_0}=\inf_{\mathbf{q}\in \R^n}\left(\fint_{B_\rho(x)}|\nabla f-\mathbf{q}|^{\gamma_0}\,dx\right)^{1/\gamma_0}.
				\end{equation*}
				We denote $\mathbf{q}_{x,\rho}=\mathbf{q}_{x,\rho}(u)$ and
				\begin{equation*}
					\phi(x,\rho)=\inf_{\mathbf{q}\in \R^n}\left(\fint_{B_\rho(x)}|\nabla u-\mathbf{q}|^{\gamma_0}\,dx\right)^{1/\gamma_0}.
				\end{equation*}
				Since
				$$
				|\mathbf{q}_{x,\rho}-\nabla u(x)|^{\gamma_0}\leq C|\mathbf{q}_{x,\rho}-\nabla u(\sigma)|^{\gamma_0}+C|\nabla u(\sigma)-\nabla u(x)|^{\gamma_0},
				$$
				by taking the average over $\sigma \in B_{ \rho}(x)$ and then taking the $\gamma_0$-th root, we obtain
				\begin{align*}
					&|\mathbf{q}_{x,\rho}-\nabla u(x)|\leq C\phi(x, \rho)+C\Big(\fint_{B_{ \rho}(x)}|\nabla u(\sigma)-\nabla u(x)|^{\gamma_0}\,d\sigma\Big)^{1/\gamma_0}.
				\end{align*}
				Therefore, from the definition of $\phi$ and the fact that $\gamma_0\in (0,1 )$, we obtain that
				\begin{equation}\label{limit}
					\lim_{\rho\rightarrow 0}\mathbf{q}_{x,\rho}=\nabla u(x)
				\end{equation}
				holds for any Lebesgue point $x\in \Omega$ of the vector-valued function $\nabla u$.
				\begin{proposition}\label{prop1}
					Suppose that $u\in W^{1,G}_{\mathrm{loc}}(\Omega)$ is a solution of \eqref{eq-main}.
					Then there exist constants $\alpha\in(0,1)$ and $\gamma_0\in(0,1)$, as in
					Lemmas~\ref{lem-dong21} and \ref{lem-32}, such that the following holds.
					For every $\varepsilon\in(0,1)$ there exist constants
					$C_\varepsilon>0$ and $C>0$, where $C_\varepsilon$ depends only on
					$\varepsilon, n,i_a,s_a,\lambda,\Lambda$ and $\gamma_0$, and $C$ depends
					only on the structural data $n,i_a,s_a,\lambda,\Lambda$ and $\gamma_0$,
					such that for every ball $B_{2r}(x_0)\Subset\Omega$, one has
					\begin{equation}\label{ineq:prop}
						\begin{aligned}
							\phi(x_0,\varepsilon r)
							\le\;&
							C\,\varepsilon^{\alpha}\phi(x_0,r)
							+ C_\varepsilon\, g^{-1}\!\left(\frac{|\mu|(B_{2r}(x_0))}{r^{\,n-1}}\right)\\
							&\quad
							+ C_\varepsilon
							\Bigg(g^{-1}\!\left(\frac{|\mu|(B_{2r}(x_0))}{r^{\,n-1}}\right)\Bigg)^{i_a}
							\fint_{B_{2r}(x_0)}(|\nabla u|+1)^{1-i_a}\,dx\\
							&\quad
							+ C_\varepsilon\,\omega^{\frac{2i_a}{(1+i_a)^2}}(r)
							\left(\fint_{B_{2r}(x_0)}(|\nabla u|+1)^{1-i_a}\,dx\right)^{\!\frac{1}{1-i_a}}.
						\end{aligned}
					\end{equation}
				\end{proposition}
				\begin{proof}
					By Lemma~\ref{lem-dong21} and the definition of $\mathbf{q}_{x,\rho}(\cdot)$, we get
					\begin{equation}\label{ineq:phi}
						\begin{split}
							&\left(\fint_{B_{\varepsilon r}(x_0)}|\nabla u-\mathbf{q}_{x_0,\varepsilon r}(u)|^{\gamma_0}\,dx\right)^{\frac 1{\gamma_0}}
							\le \left(\fint_{B_{\varepsilon r}(x_0)}|\nabla u-\mathbf{q}_{x_0,\varepsilon r}(v)|^{\gamma_0}\,dx\right)^{\frac 1{\gamma_0}}\\
							&\quad \leq C\left(\fint_{B_{\varepsilon r}(x_0)}|\nabla v-\mathbf{q}_{x_0,\varepsilon r}(v)|^{\gamma_0}\,dx\right)^{\frac 1{\gamma_0}}+C\left(\fint_{B_{\varepsilon r}(x_0)}|\nabla u-\nabla v|^{\gamma_0}\,dx\right)^{\frac 1{\gamma_0}}\\
							&\quad \leq C\varepsilon^{\alpha}\left(\fint_{B_r(x_0)}|\nabla v-\mathbf{q}_{x_0,r}(v)|^{\gamma_0}\,dx\right)^{\frac 1{\gamma_0}}+ C\varepsilon^{-\frac n{\gamma_0}}\left(\fint_{B_r(x_0)}|\nabla u-\nabla v|^{\gamma_0}\,dx\right)^{\frac 1{\gamma_0}}\\
							&\quad\leq C\varepsilon^{\alpha}\left(\fint_{B_r(x_0)}|\nabla v-\mathbf{q}_{x_0,r}(u)|^{\gamma_0}\,dx\right)^{\frac 1{\gamma_0}}+ C\varepsilon^{-\frac n{\gamma_0}}\left(\fint_{B_r(x_0)}|\nabla u-\nabla v|^{\gamma_0}\,dx\right)^{\frac 1{\gamma_0}}\\
							&\quad \leq C\varepsilon^{\alpha}\left(\fint_{B_r(x_0)}|\nabla u-\mathbf{q}_{x_0,r}(u)|^{\gamma_0}\,dx\right)^{\frac 1{\gamma_0}}+ C\varepsilon^{-\frac n{\gamma_0}}\left(\fint_{B_r(x_0)}|\nabla u-\nabla v|^{\gamma_0}\,dx\right)^{\frac 1{\gamma_0}}.
						\end{split}
					\end{equation}
					In addition, using~\eqref{eq-0} and~\eqref{eq-vw}, we deduce 
					\begin{equation}\label{eq-321}
						\begin{aligned}
							\left(\fint_{B_{r}(x_0)} |\nabla w-\nabla v|^{\gamma_0}\,dx\right)^{\frac{1}{\gamma_0}} &\leq \left(\fint_{B_r(x_0)}|\nabla w-\nabla v|^{1+i_a}\,dx\right)^{\frac{1}{1+i_a}} \\
							&\leq \left( \fint_{B_r(x_0)}G(|\nabla w-\nabla v|)\,dx\right)^{\frac{1}{1+i_a}}\\
							&\leq C\omega^{\frac{2i_a}{(1+i_a)^2}}(r)\left(\fint_{B_{r}(x_0)}G(|\nabla w|)\,dx\right)^{\frac{1}{1+i_a}}.
						\end{aligned}
					\end{equation}
					If $\gamma_0 \le s_a<1$, then it follows from \eqref{eq-RH} that
					\begin{equation}\label{eq-gamma-sa}
						\fint_{B_{r}(x_0)}G(|\nabla w|)\,dx\leq 
						\left(\fint_{B_{2r}(x_0)} G^{\frac{\gamma_0}{s_a}}(|\nabla w|)\,dx\right)^{\frac{s_a}{\gamma_0}} 
						\le
						C \left(\fint_{B_{2r}(x_0)} (|\nabla w|+1)^{\gamma_0}\,dx\right)^{\frac{1}{\gamma_0}}.
					\end{equation}
					If $s_a \le \gamma_0< 1$, the above inequality is obvious.
					Consequently, ~\eqref{eq-gamma-sa} holds for all $\gamma_0\in(0,1)$. Thus, from~\eqref{eq-321} and ~\eqref{eq-gamma-sa},  we obtain 
					\begin{equation}\label{eq-es}
						\left(\fint_{B_{r}(x_0)} |\nabla w-\nabla v|^{\gamma_0}\,dx\right)^{\frac{1}{\gamma_0}} \leq C\omega^{\frac{2i_a}{(1+i_a)^2}}(r)\left(\fint_{B_{2r}(x_0)} (|\nabla w|+1)^{\gamma_0}\,dx\right)^{\frac{1}{\gamma_0}}.
					\end{equation}
				 Combining (\ref{eq-es}) and  $\omega(r)\leq1$ yields
					\begin{equation}\label{ineq:u-v}
						\begin{aligned}
							&\left(	\fint_{B_r(x_0)}|\nabla u-\nabla v|^{\gamma_0}\,dx\right)^{\frac{1}{\gamma_0}}\\
							&\quad\leq C\left(\fint_{B_r(x_0)}|\nabla u-\nabla w|^{\gamma_0}\,dx\right)^{\frac{1}{\gamma_0}}+C\left(\fint_{B_r(x_0)}|\nabla w-\nabla v|^{\gamma_0}\,dx\right)^{\frac{1}{\gamma_0}}\\
							&\quad \leq C\left(\fint_{B_{2r}(x_0)}|\nabla u-\nabla w|^{\gamma_0}\,dx\right)^{\frac{1}{\gamma_0}}+ 
							C\omega^{\frac{2i_a}{(1+i_a)^2}}(r)\left(\fint_{B_{2r}(x_0)} (|\nabla w|+1)^{\gamma_0}\,dx\right)^{\frac{1}{\gamma_0}}	\\
							&\quad \leq C\left(\fint_{B_{2r}(x_0)}|\nabla u-\nabla w|^{\gamma_0}\,dx\right)^{\frac{1}{\gamma_0}}+ 
							C\omega^{\frac{2i_a}{(1+i_a)^2}}(r)\left(\fint_{B_{2r}(x_0)} (|\nabla u|+1)^{\gamma_0}\,dx\right)^{\frac{1}{\gamma_0}}.
						\end{aligned}
					\end{equation}
			It follows 	from \eqref{ineq:phi} and \eqref{ineq:u-v} that 
					\begin{equation}\label{ineq:phi2}
						\begin{aligned}
							\phi(x_0,\varepsilon r)&\leq C\varepsilon^{\alpha}\phi(x_0,r)+C_\varepsilon \left(\fint_{B_{2r}(x_0)}|\nabla u-\nabla w|^{\gamma_0}\right)^{1/\gamma_0}\\
							&\quad +C_\varepsilon\omega^{\frac{2i_a}{(1+i_a)^2}}(r)\left(\fint_{B_{2r}(x_0)} (|\nabla u|+1)^{\gamma_0}\,dx\right)^{\frac{1}{\gamma_0}}.
						\end{aligned}
					\end{equation}
					Then we can apply Lemma \ref{lem-32} to bound the second term on the right-hand side of (\ref{ineq:phi2}) to conclude the proof.
				\end{proof}

				Now we are ready to prove Theorem \ref{thm:int}.
				\begin{proof}[Proof of Theorem \ref{thm:int}]
					Let $x_0$ be a Lebesgue point of $\nabla u$ and assume that $B_{2R}(x_0)\Subset\Omega$.
					
					Since $i_a\in\big(\frac{n-1}{2n-1},1\big)$, we set $\gamma_0:=1-i_a\in(0,\frac{n}{2n-1})$ and apply Lemma~\ref{lem-32} with this choice. 
					Fix $\varepsilon\in(0,1/4)$ sufficiently small such that $C\varepsilon^\alpha\le\frac14$, where $C$ is the constant in \eqref{ineq:prop}.

					For each integer $j\ge0$, set
					\[
					r_j:=\varepsilon^jR, \qquad B^j:=B_{2r_j}(x_0),
					\]
					and introduce the quantities
					\[
					T_j:=\left(\fint_{B^j}(|\nabla u|+1)^{\gamma_0}\,dx\right)^{1/\gamma_0},
					\quad
					\phi_j:=\phi(x_0,r_j),
					\quad
					\mathbf q_j:=\mathbf q_{x_0,r_j}.
					\]
					Applying \eqref{ineq:prop} with $r_j$ yields
					\begin{align*}
						\phi_{j+1}
						\le &\;\frac14\,\phi_j
						+ C\,g^{-1}\!\left(\frac{|\mu|(B^j)}{r_j^{\,n-1}}\right)
						+ C\!\left(g^{-1}\!\left(\frac{|\mu|(B^j)}{r_j^{\,n-1}}\right)\right)^{i_a} T_j^{1-i_a} \\
						&\; + C(\omega(r_j))^{\frac{2i_a}{(1+i_a)^2}} T_j.
					\end{align*}
					
					Let $j_0$ and $m$ be positive integers  to be specified later such that $ j_0\le m$. Summing the above inequality
					from $j=j_0$ to $j=m$ yields
					\begin{equation}\label{sum:int-re-full}
						\begin{aligned}
							\sum_{j=j_0}^{m+1}\phi_j
							\le &\; C\phi_{j_0}
							+ C\sum_{j=j_0}^{m} g^{-1}\!\left(\frac{|\mu|(B^j)}{r_j^{\,n-1}}\right) \\
							&\; + C\sum_{j=j_0}^{m}
							\left(g^{-1}\!\left(\frac{|\mu|(B^j)}{r_j^{\,n-1}}\right)\right)^{i_a} T_j^{1-i_a}
							+ C\sum_{j=j_0}^{m}(\omega(r_j))^{\frac{2i_a}{(1+i_a)^2}}T_j.
						\end{aligned}
					\end{equation}
					Note that, for every $x\in B_{r_{j+1}}(x_0)$, we have
					\[
					|\mathbf q_{j+1}-\mathbf q_j|^{\gamma_0}
					\le |\mathbf q_{j+1}-\nabla u(x)|^{\gamma_0}
					+|\nabla u(x)-\mathbf q_j|^{\gamma_0}.
					\]
					Integrating over $B_{r_{j+1}}(x_0)$ and taking the $\gamma_0$-th root, we obtain
					\[
					|\mathbf q_{j+1}-\mathbf q_j|
					\le C(\phi_{j+1}+\phi_j).
					\]
					Iterating this inequality from $j=j_0$ to $j=m$ gives
					\[
					|\mathbf q_{m+1}-\mathbf q_{j_0}|
					\le C\sum_{j=j_0}^{m+1}\phi_j.
					\]
					Combining this estimate with \eqref{sum:int-re-full}, we infer that 
					\begin{equation}\label{ineq:q1-re-full}
						\begin{aligned}
							|\mathbf q_{m+1}|+\sum_{j=j_0}^{m+1}\phi_j
							\le &\; C\phi_{j_0}+|\mathbf q_{j_0}| + C\sum_{j=j_0}^{m} g^{-1}\!\left(\frac{|\mu|(B^j)}{r_j^{\,n-1}}\right) \\
							&\; + C\sum_{j=j_0}^{m}
							\left(g^{-1}\!\left(\frac{|\mu|(B^j)}{r_j^{\,n-1}}\right)\right)^{i_a} T_j^{1-i_a}
							+ C\sum_{j=j_0}^{m}(\omega(r_j))^{\frac{2i_a}{(1+i_a)^2}}T_j.
						\end{aligned}
					\end{equation}
					By the definition of $\phi_{j_0}$, it follows that
					$$\phi_{j_0}
					\le C\left(\fint_{B^{j_0}}|\nabla u|^{\gamma_0}\,dx\right)^{1/\gamma_0}
					\le C T_{j_0}.$$
					Moreover, for every $x\in B_{r_{j_0}}(x_0)$,
					\[
					|\mathbf q_{j_0}|^{\gamma_0}
					\le |\nabla u(x)-\mathbf q_{j_0}|^{\gamma_0}
					+|\nabla u(x)|^{\gamma_0}.
					\]
					Averaging over $B_{r_{j_0}}(x_0)$ and taking the $\gamma_0$-th root yields
					\[
					|\mathbf q_{j_0}|
					\le C\phi_{j_0}
					+ C\left(\fint_{B^{j_0}}|\nabla u|^{\gamma_0}\,dx\right)^{1/\gamma_0}
					\le C T_{j_0}.
					\]
					Substituting these bounds into \eqref{ineq:q1-re-full}, we arrive at
					\begin{equation}\label{ineq:q-re-full}
						\begin{aligned}
							|\mathbf q_{m+1}|+\sum_{j=j_0}^{m+1}\phi_j
							\le &\; CT_{j_0}
							+ C\sum_{j=j_0}^{m} g^{-1}\!\left(\frac{|\mu|(B^j)}{r_j^{\,n-1}}\right) \\
							&\; + C\sum_{j=j_0}^{m}
							\left(g^{-1}\!\left(\frac{|\mu|(B^j)}{r_j^{\,n-1}}\right)\right)^{i_a} T_j^{1-i_a}
							+ C\sum_{j=j_0}^{m}(\omega(r_j))^{\frac{2i_a}{(1+i_a)^2}}T_j .
						\end{aligned}
					\end{equation}
					By the Dini condition \eqref{eq:dini} and the fact that $r_j\to0$ as
					$j\to\infty$, there exists $j_0$ sufficiently large such that
					\[
					\sum_{j=j_0}^\infty(\omega(r_j))^{\frac{2i_a}{(1+i_a)^2}}
					\]
					is arbitrarily small.
					Moreover, by the monotonicity of $g^{-1}$ and a Riemann sum argument, together with the comparison principle for Riemann integrals, we obtain
					\begin{equation}
						\begin{split}
							\mathcal{W}_{i_a,g}^{\mu}(x,2R)
							&= \int_0^{2R}
							\left(g^{-1}\!\left(\frac{|\mu|(B(x,\rho))}{\rho^{n-1}}\right)\right)^{i_a}
							\frac{d\rho}{\rho} \\
							&= \sum_{j=0}^\infty
							\int_{2r_{j+1}}^{2r_j}
							\left(g^{-1}\!\left(\frac{|\mu|(B(x,\rho))}{\rho^{n-1}}\right)\right)^{i_a}
							\frac{d\rho}{\rho} \\
							&\ge \sum_{j=0}^\infty
							\left(g^{-1}\!\left(\frac{|\mu|(B^{j+1})}{(2r_j)^{n-1}}\right)\right)^{i_a}
							\ln\frac{1}{\varepsilon} \\
							&\ge C \varepsilon^{n-1}\ln\frac{1}{\varepsilon}
							\sum_{j=1}^\infty
							\left(g^{-1}\!\left(\frac{|\mu|(B^j)}{r_j^{n-1}}\right)\right)^{i_a}.
						\end{split}
					\end{equation}
					In addition, we deduce that 
					\begin{align*}
						\sum_{j=j_0}^{m}
						\left(g^{-1}\!\left(\frac{|\mu|(B^j)}{r_j^{\,n-1}}\right)\right)^{i_a}
						\le
						C\int_0^{2r_{j_0}}
						\left(g^{-1}\!\left(\frac{|\mu|(B_\rho(x_0))}{\rho^{n-1}}\right)\right)^{i_a}
						\frac{d\rho}{\rho}.
					\end{align*}
					and 
					\begin{equation}\label{eqj0}
						\begin{split}			\sum_{j=j_0}^{m}
							g^{-1}\!\left(\frac{|\mu|(B^j)}{r_j^{\,n-1}}\right)& =\sum_{j=j_0}^{m}
							\left(g^{-1}\!\left(\frac{|\mu|(B^j)}{r_j^{\,n-1}}\right)\right)^{i_a\cdot \frac{1}{i_a}}\\
							&\le \left(  \sum_{j=j_0}^{m}
							\left(g^{-1}\!\left(\frac{|\mu|(B^j)}{r_j^{\,n-1}}\right)\right)^{i_a}    \right)^{\frac{1}{i_a}}\\
							&		\le
							C\left(\int_0^{2r_{j_0}}
							\left(g^{-1}\!\left(\frac{|\mu|(B_\rho(x_0))}{\rho^{n-1}}\right)\right)^{i_a}
							\frac{d\rho}{\rho}\right)^{\frac{1}{i_a}}.
					\end{split}		\end{equation}
					
					To prove \eqref{ineq:int} at $x=x_0$, it is sufficient to show that 
					\begin{equation}\label{ineq:result}
						|\nabla u(x_0)|\leq CT_{j_0}+ C\left(\int_0^{2r_{j_0}}\left(g^{-1}\left(\frac{|\mu|(B_\rho(x_0))}{\rho^{n-1}}\right)\right)^{i_a} \frac{d\rho}{\rho}\right)^{\frac{1}{i_a}}. 
					\end{equation}

					We distinguish the following cases.

					{\em Case 1.} If $|\nabla u(x_0)|\leq T_{j_0}$, then (\ref{ineq:result}) easily follows.

					{\em Case 2.} If $ T_j< |\nabla u(x_0)|, \ \forall j_0\leq j\leq j_1$, and $|\nabla u(x_0)|\leq T_{j_1+1}$, then since $\gamma_0=1-i_a<1$, we have
					\begin{equation}\label{ineq:grad}
						\begin{aligned} 
							|\nabla u(x_0)|&\leq\left(\fint_{B^{j_1+1}}(|\nabla u|+1)^{\gamma_0}\,dx\right)^{1/\gamma_0}\\ 
							&\leq 2^{\frac{1}{\gamma_0}}\left(\fint_{B^{j_1+1}}|\nabla u|^{\gamma_0}\,dx\right)^{1/\gamma_0}+2^{\frac{1}{\gamma_0}}\\ &\leq 2^{\frac{1}{\gamma_0}}(2\varepsilon)^{-\frac{n}{\gamma_0}} \left(\fint_{B_{r_{j_1}}(x_0)}|\nabla u|^{\gamma_0}\,dx\right)^{1/\gamma_0}+2^{\frac{1}{\gamma_0}}\\ 
							&\leq 4^\frac{1}{\gamma_0}(2\varepsilon)^{-\frac{n}{\gamma_0}}(\phi_{j_1}+|\mathbf{q}_{j_1}|) +2^{\frac{1}{\gamma_0}}, 
					\end{aligned} \end{equation}
					where the last inequality follows from the definitions of $\phi_{j_1}$ and $\mathbf{q}_{j_1}$.
					
					Now applying \eqref{ineq:q-re-full} with $m=j_1-1$ and using \eqref{eqj0}, from \eqref{ineq:grad} we get
					\begin{align*} 
						|\nabla u(x_0)|\leq&CT_{j_0} +C\left(\int_0^{2r_{j_0}}\left(g^{-1}\left(\frac{|\mu|(B_\rho(x_0))}{\rho^{n-1}}\right)\right)^{i_a} \frac{d\rho}{\rho}\right)^\frac{1}{i_a}\\ 
						&+C\int_0^{2r_{j_0}} \left(g^{-1}\left(\frac{|\mu|(B_\rho(x_0))}{\rho^{n-1}} \right)\right)^{i_a}\frac{d\rho}{\rho}\cdot|\nabla u(x_0)|^{1-i_a}\\ &+C\sum_{j=j_0}^m(\omega(r_j))^{\frac{2i_a}{(1+i_a)^2}}|\nabla u(x_0)|+2^\frac{1}{\gamma_0}.
					\end{align*} 
					Hence using the Dini-type condition \eqref{eq:dini} and Young's inequality, we find 
					\begin{align*} |\nabla u(x_0)|\leq CT_{j_0}+C \left( \int_0^{2r_{j_0}} \left(g^{-1}\left(\frac{|\mu|(B_\rho(x_0))}{\rho^{n-1}} \right)\right)^{i_a}\frac{d\rho}{\rho}\right)^{\frac{1}{i_a}} +\frac{1}{5}|\nabla u(x_0)|. 
					\end{align*} 
					This implies (\ref{ineq:result}) as desired.
					
					{\em Case 3.} If $T_j<|\nabla u(x_0)|$ for any $j\geq j_0$, then from \eqref{ineq:q-re-full} and \eqref{eqj0} and Young's inequality, we have for any $m>j_0$,
					\begin{align*}
						|\mathbf{q}_{m+1}|\leq& CT_{j_0}+C\left( \int_0^{2r_{j_0}} \left(g^{-1}\left(\frac{|\mu|(B_\rho(x_0))}{\rho^{n-1}}\right)\right)^{i_a}\frac{d\rho}{\rho}\right)^{\frac{1}{i_a}} \\ 
						&+C \int_0^{2r_{j_0}}\left(g^{-1}\left(\frac{|\mu|(B_\rho(x_0))}{\rho^{n-1}} \right)\right)^{i_a}\frac{d\rho}{\rho}\cdot|\nabla u(x_0)|^{1-i_a}+C\sum_{j=j_0}^m(\omega(r_j))^{\frac{2i_a}{i_a+1}}|\nabla u(x_0)|\\ 
						\leq& CT_{j_0}+C\left( \int_0^{2r_{j_0}} \left(g^{-1}\left(\frac{|\mu|(B_\rho(x_0))}{\rho^{n-1}} \right)\right)^{i_a}\frac{d\rho}{\rho}\right)^{\frac{1}{i_a}}+\frac{1}{5}|\nabla u(x_0)|.
					\end{align*} 
					Letting $m\to \infty$ and using \eqref{limit}, we get 
					\begin{align*}
						|\nabla u(x_0)|\leq& CT_{j_0}+C\left( \int_0^{2r_{j_0}} \left(g^{-1}\left(\frac{|\mu|(B_\rho(x_0))}{\rho^{n-1}} \right)\right)^{i_a}\frac{d\rho}{\rho}\right)^{\frac{1}{i_a}}+\frac{1}{5}|\nabla u(x_0)|.
					\end{align*}
					Then we deduce (\ref{ineq:result}). The proof is completed. 
				\end{proof}
				
				\section{Interior Lipschitz continuity of the gradient}\label{sec4}
				
				This section is devoted to the proof of the interior Lipschitz bound stated in Theorem~\ref{thm:int:lip}.
				Let $\alpha\in(0,1)$ be the exponent appearing in Lemma~\ref{lem-dong21} and fix any
				$\alpha_1\in(0,\alpha)$. Let $R\in(0,1]$ be such that $B_R(x_0)\Subset\Omega$.
				We choose a parameter
				\[
				\varepsilon=\varepsilon(n,i_a,s_a,\Lambda,\lambda,\gamma_0,\alpha,\alpha_1)\in(0,1)
				\]
				small enough so that
				\[
				C\,\varepsilon^{\alpha-\alpha_1}<1
				\qquad\text{and}\qquad
				\varepsilon^{\alpha_1}<\frac14,
				\]
				where $C$ denotes the constant appearing in \eqref{ineq:prop}.
				
				As a consequence of Proposition~\ref{prop1}, for every ball
				$B_{2r}(x)\Subset B_R(x_0)$ we have the decay inequality
				\begin{equation}\label{eq7.15-new}
					\begin{aligned}
						\phi(x,\varepsilon r)
						\leq\;& \varepsilon^{\alpha_1}\phi(x,r)
						+ C\, g^{-1}\!\left(\frac{|\mu|(B_{2r}(x))}{r^{\,n-1}}\right) \\
						& + C\!\left(g^{-1}\!\left(\frac{|\mu|(B_{2r}(x))}{r^{\,n-1}}\right)\right)^{i_a}
						\big(\|\nabla u\|_{L^\infty(B_{2r}(x))}+1\big)^{1-i_a} \\
						& + C\,\omega^{\frac{2i_a}{(1+i_a)^2}}(r)
						\big(\|\nabla u\|_{L^\infty(B_{2r}(x))}+1\big) .
					\end{aligned}
				\end{equation}
				
				For convenience, we introduce the auxiliary quantities
				\begin{equation}\label{def:aux-new}
					h(x,r)
					:= \left(g^{-1}\!\left(\frac{|\mu|(B_r(x))}{r^{\,n-1}}\right)\right)^{i_a},
					\qquad
					H(x,r) := h(x,r)^{1/i_a}.
				\end{equation}
				
				Iterating \eqref{eq7.15-new}, we infer that for every integer $j\ge1$,
				\begin{align*}
					\phi(x,\varepsilon^j r)
					\leq\;& \varepsilon^{\alpha_1 j}\phi(x,r)
					+ C\sum_{i=1}^j \varepsilon^{\alpha_1(i-1)} H(x,2\varepsilon^{j-i}r) \\
					& + C\sum_{i=1}^j \varepsilon^{\alpha_1(i-1)} h(x,2\varepsilon^{j-i}r)
					\big(\|\nabla u\|_{L^\infty(B_{2r}(x))}+1\big)^{1-i_a} \\
					& + C\sum_{i=1}^j \varepsilon^{\alpha_1(i-1)}
					\omega^{\frac{2i_a}{(1+i_a)^2}}(\varepsilon^{j-i}r)
					\big(\|\nabla u\|_{L^\infty(B_{2r}(x))}+1\big),
				\end{align*}
				whenever $r\in(0,R/4)$ and $B_{2r}(x)\Subset B_R(x_0)$.
				
				Collecting the above contributions, we obtain
				\begin{equation}\label{iter:int1-new}
					\begin{aligned}
						\phi(x,\varepsilon^j r)
						\leq\;& \varepsilon^{\alpha_1 j}\phi(x,r)
						+ C\,\widetilde H(x,2\varepsilon^j r)
						+ C\,\widetilde h(x,2\varepsilon^j r)
						\big(\|\nabla u\|_{L^\infty(B_{2r}(x))}+1\big)^{1-i_a} \\
						& + C\,\widetilde\omega^{\frac{2i_a}{(1+i_a)^2}}(\varepsilon^j r)
						\big(\|\nabla u\|_{L^\infty(B_{2r}(x))}+1\big),
					\end{aligned}
				\end{equation}
				where we defined 
				\begin{equation}\label{eq6.42-new}
					\begin{aligned}
						\widetilde H(x,t)
						&:= \sum_{i=1}^\infty \varepsilon^{\alpha_1 i}
						\Big( H(x,\varepsilon^{-i}t)\,[\varepsilon^{-i}t\le R/2]
						+ H(x,R/2)\,[\varepsilon^{-i}t>R/2] \Big),\\
						\widetilde h(x,t)
						&:= \sum_{i=1}^\infty \varepsilon^{\alpha_1 i}
						\Big( h(x,\varepsilon^{-i}t)\,[\varepsilon^{-i}t\le R/2]
						+ h(x,R/2)\,[\varepsilon^{-i}t>R/2] \Big),\\
						\widetilde\omega^{\frac{2i_a}{(1+i_a)^2}}(t)
						&:= \sum_{i=1}^\infty \varepsilon^{\alpha_1 i}
						\Big( \omega^{\frac{2i_a}{(1+i_a)^2}}(\varepsilon^{-i}t)\,[\varepsilon^{-i}t\le R/2]
						+ \omega^{\frac{2i_a}{(1+i_a)^2}}(R/2)\,[\varepsilon^{-i}t>R/2] \Big).
					\end{aligned}
				\end{equation}
				Here $[\,\cdot\,]$ denotes the Iverson bracket, that is, $[\,P\,]=1$ if $P$ is true and $[\,P\,]=0$ if $P$ is false.
				
				Before turning to the proof of Theorem~\ref{thm:int:lip}, we first derive an auxiliary
				iteration result.
				
				\begin{lemma}\label{lem:iter1}
					Let $B_{2r}(x)\Subset B_R(x_0)\subset\Omega$ with $r\le R/4$.
					There exists a constant $C$, depending only on
					$\varepsilon$, $n$, $i_a$, $s_a$, $\Lambda$, $\lambda$, $\gamma_0$ and $\alpha_1$,
					such that for every $\rho\in(0,r]$, we have 
					\begin{equation}\label{iter:int2}
						\begin{aligned}
							\phi(x,\rho)
							\leq\;& C\left(\frac{\rho}{r}\right)^{\alpha_1}\phi(x,r)
							+ C\,\widetilde H(x,2\rho) \\
							& + C\,\widetilde h(x,2\rho)
							\big(\|\nabla u\|_{L^\infty(B_{2r}(x))}+1\big)^{1-i_a} \\
							& + C\,\widetilde\omega^{\frac{2i_a}{(1+i_a)^2}}(\rho)
							\big(\|\nabla u\|_{L^\infty(B_{2r}(x))}+1\big),
						\end{aligned}
					\end{equation}
					and, 
					\begin{equation}\label{sum1:phi}
						\begin{aligned}
							\sum_{j=0}^\infty \phi(x,\varepsilon^j\rho)
							\leq\;& C\left(\frac{\rho}{r}\right)^{\alpha_1}\phi(x,r)
							+ C\int_0^{\rho}\frac{\widetilde H(x,t)}{t}\,dt \\
							& + C\big(\|\nabla u\|_{L^\infty(B_{2r}(x))}+1\big)^{1-i_a}
							\int_0^{\rho}\frac{\widetilde h(x,t)}{t}\,dt \\
							& + C\big(\|\nabla u\|_{L^\infty(B_{2r}(x))}+1\big)
							\int_0^{\rho}\frac{\widetilde\omega^{\frac{2i_a}{(1+i_a)^2}}(t)}{t}\,dt.
						\end{aligned}
					\end{equation}
				\end{lemma}
				
				\begin{proof}
					Fix $\rho\in(0,r]$ and choose an integer $j\ge0$ such that
					\[
					\varepsilon^{j+1}<\frac{\rho}{r}\le\varepsilon^j .
					\]
					Applying \eqref{iter:int1-new} with $\varepsilon^{-j}\rho$ in place of $r$, we obtain
					\begin{equation}\label{eq-a}
						\begin{split}
						\phi(x,\rho)
						\leq\;& \varepsilon^{\alpha_1 j}\phi(x,\varepsilon^{-j}\rho)
						+ C\,\widetilde H(x,2\rho)
						+ C\,\widetilde h(x,2\rho)
						\big(\|\nabla u\|_{L^\infty(B_{2\varepsilon^{-j}\rho}(x))}+1\big)^{1-i_a} \\
						& + C\,\widetilde\omega^{\frac{2i_a}{(1+i_a)^2}}(\rho)
						\big(\|\nabla u\|_{L^\infty(B_{2\varepsilon^{-j}\rho}(x))}+1\big).
						\end{split}
					\end{equation}
					Since $\varepsilon^{-j}\rho\le r$, it follows that
					\[
					\|\nabla u\|_{L^\infty(B_{2\varepsilon^{-j}\rho}(x))}
					\le \|\nabla u\|_{L^\infty(B_{2r}(x))}.
					\]
			 By the choice of $j$,
					$\varepsilon^{\alpha_1 j}\le \varepsilon^{-\alpha_1}(\rho/r)^{\alpha_1}$.
					Substituting these bounds into \eqref{eq-a} yields
					\eqref{iter:int2}.
					
					Applying \eqref{iter:int2} with $\varepsilon^j\rho$ in place of $\rho$
					and summing over $j\ge0$, we obtain
					\begin{align*}
						\sum_{j=0}^\infty \phi(x,\varepsilon^j\rho)
						\leq\;& C\left(\frac{\rho}{r}\right)^{\alpha_1}\phi(x,r)
						+ C\sum_{j=0}^\infty \widetilde H(x,2\varepsilon^j\rho) \\
						& + C\big(\|\nabla u\|_{L^\infty(B_{2r}(x))}+1\big)^{1-i_a}
						\sum_{j=0}^\infty \widetilde h(x,2\varepsilon^j\rho) \\
						& + C\big(\|\nabla u\|_{L^\infty(B_{2r}(x))}+1\big)
						\sum_{j=0}^\infty \widetilde\omega^{\frac{2i_a}{(1+i_a)^2}}(2\varepsilon^j\rho).
					\end{align*}
					Using the definition of $\widetilde H$, $\widetilde h$ and $\widetilde\omega$, one easily checks that there exist constants $c_1,c_2>0$, depending only on $\varepsilon$, $n$, $i_a$, $s_a$ and $\alpha_1$, such that \[ c_1 f(t)\le f(s)\le c_2 f(t) \qquad\text{for all } 0<\varepsilon t\le s\le t,\] for any $f\in\{\widetilde H(x,\cdot),\widetilde h(x,\cdot),\widetilde\omega(\cdot)\}$. Consequently, by the standard comparison between geometric sums and Riemann integrals, we obtain \eqref{sum1:phi}.
				\end{proof}
				
				Recall the definition of $\mathbf{q}_{x,\rho}$ from Section~\ref{sec3}. 
				For any $z\in B_{\varepsilon\rho}(x)$, we have
				\[
				|\mathbf{q}_{x,\varepsilon \rho}-\mathbf{q}_{x,\rho}|^{\gamma_0}
				\leq |\nabla u(z)-\mathbf{q}_{x,\rho}|^{\gamma_0}
				+|\nabla u(z)-\mathbf{q}_{x,\varepsilon \rho}|^{\gamma_0}.
				\]
				Averaging over $B_{\varepsilon\rho}(x)$ and taking the $\gamma_0$-th root yield
				\[
				|\mathbf{q}_{x,\varepsilon \rho}-\mathbf{q}_{x,\rho}|
				\leq C\phi(x,\varepsilon \rho)+C\phi(x,\rho).
				\]
				Iterating this estimate gives
				\[
				|\mathbf{q}_{x,\varepsilon^j\rho}-\mathbf{q}_{x,\rho}|
				\leq C\sum_{i=0}^j \phi(x,\varepsilon^i\rho).
				\]
				Therefore, by \eqref{limit},
				\begin{equation}\label{ineq:diff}
					|\nabla u(x)-\mathbf{q}_{x,\rho}|
					\leq C\sum_{j=0}^\infty \phi(x,\varepsilon^j\rho),
				\end{equation}
				for every Lebesgue point $x\in\Omega$ of $\nabla u$.
				
				Now we are ready to prove the interior Lipschitz estimate.
				\begin{proof}[Proof of Theorem~\ref{thm:int:lip}]
					Fix a point $x_0\in\Omega$ and assume that $B_R(x_0)\Subset\Omega$ for some $R\in(0,1]$. 
					Throughout the proof we suppose that the truncated Wolff potential is locally bounded, namely
					\begin{equation}\label{eq:assump-wolff}
						\big\|\mathbf{W}_{i_a,g}^{R}(|\mu|)\big\|_{L^\infty(B_R(x_0))}<\infty .
					\end{equation}
					The argument is divided into two steps. We first establish the desired estimate under a smoothness assumption on $u$, and then remove this restriction by an approximation procedure.
					
					\medskip
					\noindent
					\textit{Step 1. Assume that $u\in C^1(\overline{B_R(x_0)})$}.
					Let $x\in B_R(x_0)$ and $\rho\in(0,R/4]$ be such that $B_{2\rho}(x)\Subset B_R(x_0)$.
					Combining the excess-decay estimate \eqref{sum1:phi} with \eqref{ineq:diff}, we infer that
					\begin{align*}
						|\nabla u(x)-\mathbf q_{x,\rho}|
						\leq\;& C\phi(x,\rho)
						+ C\int_0^\rho \frac{\widetilde H(x,t)}{t}\,dt \\
						&+ C\big(\|\nabla u\|_{L^\infty(B_{2\rho}(x))}+1\big)^{1-i_a}
						\int_0^\rho \frac{\widetilde h(x,t)}{t}\,dt  \\
						&+ C\big(\|\nabla u\|_{L^\infty(B_{2\rho}(x))}+1\big)
						\int_0^\rho \frac{\widetilde\omega^{\frac{2i_a}{(1+i_a)^2}}(t)}{t}\,dt .
					\end{align*}
				In addition, by the definition of $\mathbf q_{x,\rho}$, we have
					\[
					|\mathbf q_{x,\rho}|
					\le C\phi(x,\rho)+C\rho^{-n/\gamma_0}\|\nabla u\|_{L^{\gamma_0}(B_\rho(x))}
					\le C\rho^{-n/\gamma_0}\|\nabla u\|_{L^{\gamma_0}(B_\rho(x))}.
					\]
					Combining the previous two inequalities yields
					\begin{align}\label{eq:local-grad-est}
						|\nabla u(x)|
						\leq\;& C\rho^{-n/\gamma_0}\|\nabla u\|_{L^{\gamma_0}(B_\rho(x))}
						+ C\int_0^\rho \frac{\widetilde H(x,t)}{t}\,dt \nonumber\\
						&+ C\big(\|\nabla u\|_{L^\infty(B_{2\rho}(x))}+1\big)^{1-i_a}
						\int_0^\rho \frac{\widetilde h(x,t)}{t}\,dt  \\
						&+ C\big(\|\nabla u\|_{L^\infty(B_{2\rho}(x))}+1\big)
						\int_0^\rho \frac{\widetilde\omega^{\frac{2i_a}{(1+i_a)^2}}(t)}{t}\,dt .\nonumber
					\end{align}
					
					Since $\widetilde{\omega}$ satisfies the Dini-type condition \eqref{eq:dini},
					we may choose $\rho_0\in(0,R/4]$ sufficiently small such that
					\[
					C\int_0^{\rho_0}\frac{\widetilde{\omega}^{\frac{2i_a}{(1+i_a)^2}}(t)}{t}\,dt
					\le 3^{-1-n/\gamma_0}.
					\]
					Thus, for any $\rho\le\rho_0$,  by Young's inequality we  obtain
					\begin{align}\label{eq:local-iter}
						|\nabla u(x)|
						\leq\;& C\rho^{-n/\gamma_0}\|\nabla u\|_{L^{\gamma_0}(B_\rho(x))}
						+ C\int_0^\rho \frac{\widetilde H(x,t)}{t}\,dt \nonumber\\
						&+ C\Big(\int_0^\rho \frac{\widetilde h(x,t)}{t}\,dt\Big)^{\frac1{i_a}}
						+ 3^{-n/\gamma_0}\big(\|\nabla u\|_{L^\infty(B_{2\rho}(x))}+1\big).
					\end{align}					
		Define $\rho_k=(1-2^{-k})R$ for $k\ge1$.  Since $\rho_{k+1}-\rho_k=2^{-k-1}R$, we have $B_{2\rho}(x)\subset B_{\rho_{k+1}}(x_0)$ for any $x\in  B_{\rho_{k}}(x_0)$ and $\rho=2^{-k-2}R$. We take $k_0$ sufficiently large such that $2^{-k_0-2}\leq \rho_0$. Then by \eqref{eq:local-iter} with $\rho=2^{-k-2}R$, we have for any $k\geq k_0$ that
					\begin{align*}
						\|\nabla u\|_{L^\infty(B_{\rho_k}(x_0))}+1
						\le\;& 3^{-n/\gamma_0}\big(\|\nabla u\|_{L^\infty(B_{\rho_{k+1}}(x_0))}+1\big) \\
						&+ C\Big(\frac{2^{k}}{R}\Big)^{n/\gamma_0}
						\|\nabla u\|_{L^{\gamma_0}(B_{\rho_{k+1}}(x_0))} \\
						&+ C\sup_{x\in B_{\rho_k}(x_0)}
						\int_0^{R/2}\frac{\widetilde H(x,t)}{t}\,dt \\
						&+ C\sup_{x\in B_{\rho_k}(x_0)}
						\Big(\int_0^{R/2}\frac{\widetilde h(x,t)}{t}\,dt\Big)^{\frac1{i_a}} .
					\end{align*}
					Multiplying the inequality by $3^{-nk/\gamma_0}$
					and summing over $k\ge k_0$, we arrive at
					\begin{align*} 
						&\sum_{k=k_0}^\infty 3^{-nk/{\gamma_0}}( \|\nabla u\|_{L^\infty(B_{\rho_k}(x_0))}+1)\\
						&\leq \sum_{k=k_0+1}^\infty 3^{-nk/{\gamma_0}}( \|\nabla u\|_{L^\infty(B_{\rho_k}(x_0))}+1) +CR^{-n/{\gamma_0}} (\|\nabla u\|_{L^{\gamma_0}(B_{R}(x_0))}+1) \\
						&\quad+C \sup_{x\in B_{R}(x_0)}\int_0^{\frac{R}{2}}\frac{\widetilde{H}(x,t)}{t}\,dt +C \sup_{x\in B_{R}(x_0)} \left(\int_0^{\frac{R}{2}}\frac{\widetilde{h}(x,t)}{t}\,dt\right)^{\frac{1}{i_a}}, 
					\end{align*} 
					where each summation is finite. By subtracting $$ \sum_{k=k_0+1}^\infty 3^{-nk/{\gamma_0}}( \|\nabla u\|_{L^\infty(B_{\rho_k}(x_0))}+1) $$ from both sides of the above inequality, we get the following $L^{\infty}$-estimate for $\nabla u$: 
					\begin{equation}\label{infty-norm1}
						\begin{aligned} 
							&\|\nabla u\|_{L^\infty(B_{R/2}(x_0))}+1\leq C R^{-n/{\gamma_0}} \|\nabla u\|_{L^{\gamma_0}(B_{R}(x_0))} \\ &\qquad+C \sup_{x\in B_{R}(x_0)}\int_0^{\frac{R}{2}}\frac{\widetilde{H}(x,t)}{t}\,dt +C \sup_{x\in B_{R}(x_0)} \left(\int_0^{\frac{R}{2}}\frac{\widetilde{h}(x,t)}{t}\,dt\right)^{\frac{1}{i_a}}.
						\end{aligned}
					\end{equation}

					Using the definitions of $\widetilde H$ and $\widetilde h$,  it is not difficult to  check that
					\[
					\int_0^{R/2}\frac{\widetilde H(x,t)}{t}\,dt
					+\Big(\int_0^{R/2}\frac{\widetilde h(x,t)}{t}\,dt\Big)^{\frac1{i_a}}
					\le C\big(\mathbf W_{i_a,g}^R(|\mu|)(x)\big)^{\frac1{i_a}} .
					\]
					By Lemma \ref{lem-32}, we know that  $\gamma_0\le1-i_a$. Thus  we  deduce from  \eqref{infty-norm1} that
					\begin{equation}\label{eq:Linf-smooth}
						\|\nabla u\|_{L^\infty(B_{R/2}(x_0))}
						\le C\big\|\mathbf W_{i_a,g}^R(|\mu|)\big\|_{L^\infty(B_R(x_0))}^{\frac1{i_a}}
						+ CR^{-\frac{n}{1-i_a}}\||\nabla u|+1\|_{L^{1-i_a}(B_R(x_0))}.
					\end{equation}
					
					\medskip
					\noindent
					\textit{Step 2. Approximation argument.}
					Let $r_1\in(0,R)$ and set $r_2=(R+r_1)/2$.
					Choose a sequence of standard mollifiers $\{\varphi_k\}$ and define
					\[
					\mu_k=\mu*\varphi_k,
					\qquad
					\cA_k(x,\xi)=(\cA(\cdot,\xi)*\varphi_k)(x)
					\quad\text{in }B_{r_2}(x_0).
					\]
					For $k$ sufficiently large, $\cA_k$ satisfies the same structural conditions as $\cA$ in $B_{r_2}(x_0)$.
					By \cite[Theorem~3]{CGZ24}, assumption \eqref{eq:assump-wolff} ensures that
					$\mu\in(W^{1,G}(B_{r_2}(x_0)))'$, and hence
					\[
					\mu_k\to\mu \quad\text{in }(W^{1,G}(B_{r_2}(x_0)))'.
					\]
					
					Let $u_k\in u+W^{1,G}_0(B_{r_2}(x_0))$ solve
					\[
					-\DIV(\cA_k(x,\nabla u_k))=\mu_k \quad\text{in }B_{r_2}(x_0).
					\]
					
					Standard energy estimates yield a uniform bound of $\{\nabla u_k\}$ in $L^G(B_{r_2}(x_0))$.
					Moreover, a monotonicity argument implies
					\[
					\nabla u_k\to\nabla u \quad\text{strongly in }L^G(B_{r_2}(x_0)),
					\]
					up to a subsequence.
					
					Since $\cA_k$ and $\mu_k$ are smooth, classical regularity theory gives
					$u_k\in C^{1,\alpha}_{\mathrm{loc}}(B_{r_2}(x_0))$.
					Therefore estimate \eqref{eq:Linf-smooth} applies to $u_k$ on $B_{r_1/2}(x_0)$:
					\[
					\|\nabla u_k\|_{L^\infty(B_{r_1/2}(x_0))}
					\le C\big\|\mathbf W_{i_a,g}^{r_1}(|\mu_k|)\big\|_{L^\infty(B_{r_1}(x_0))}^{\frac1{i_a}}
					+ Cr_1^{-\frac{n}{1-i_a}}\||\nabla u|+1\|_{L^{1-i_a}(B_{r_1}(x_0))}.
					\]
					By Fubini--Tonelli theorem and the definition of $\mu_k$, the Wolff potentials satisfy
					\[
					\big\|\mathbf W_{i_a,g}^{r_1}(|\mu_k|)\big\|_{L^\infty(B_{r_1}(x_0))}
					\le \big\|\mathbf W_{i_a,g}^{r_1}(|\mu|)\big\|_{L^\infty(B_{r_2}(x_0))}
					\]
					for all large $k$.
					Passing to the limit $k\to\infty$ and then letting $r_1\to R$, we obtain the desired Lipschitz bound
					\eqref{ineq:int:lip} at $x_0$.
				\end{proof}

				\medskip
				
				\subsection*{Acknowledgment}
				This work was supported by the National Natural Science Foundation of China (No. 12471128).

				\subsection*{Conflict of interest} The authors declare that there is no conflict of interest. We also declare that this
				manuscript has no associated data.
				
				\subsection*{Data availability} Data sharing is not applicable to this article as no datasets were generated or analysed
				during the current study.

			\end{document}